\ifdefined\XeTeXversion
\else
\pdfoutput=1
\fi
\documentclass[11pt,a4paper]{article}
\usepackage[T1]{fontenc}
\usepackage{lmodern}
\usepackage{amsmath,amssymb,amsthm,mathtools}
\usepackage[textwidth=440pt,top=27mm,bottom=27mm]{geometry}
\usepackage{microtype}
\usepackage{booktabs,array}
\usepackage{xcolor}
\usepackage{needspace}
\usepackage[colorlinks=true,linkcolor=blue!45!black,citecolor=blue!45!black,urlcolor=blue!45!black]{hyperref}
\hypersetup{pdftitle={Polynomial families of incident flags and explicit off-diagonal Ramsey graphs},pdfauthor={Brecht Verbeken}}
\newcommand{\F}{\mathbb F}
\newcommand{\A}{\mathbb A}
\newcommand{\PP}{\mathbb P}
\newcommand{\Ccal}{\mathcal C}
\newcommand{\Ical}{\mathcal I}
\newcommand{\Wcal}{\mathcal W}
\DeclareMathOperator{\cdeg}{cdeg}
\DeclareMathOperator{\rank}{rank}
\DeclareMathOperator{\Tr}{Tr}
\DeclareMathOperator{\spanop}{span}
\newtheorem{theorem}{Theorem}[section]
\newtheorem{proposition}[theorem]{Proposition}
\newtheorem{lemma}[theorem]{Lemma}
\newtheorem{corollary}[theorem]{Corollary}
\theoremstyle{definition}

\theoremstyle{remark}

\numberwithin{equation}{section}
\begin{document}
\title{Polynomial families of incident flags\\and explicit off-diagonal Ramsey graphs}
\author{Brecht Verbeken\textsuperscript{1,2}\\[0.6em]
\parbox{0.96\textwidth}{\centering\small
\textsuperscript{1}Department of Business Technology and Operations,\\
Data Analytics Laboratory, Vrije Universiteit Brussel (VUB),\\
Pleinlaan 2, 1050 Brussels, Belgium.\\[0.35em]
\textsuperscript{2}imec-SMIT, Vrije Universiteit Brussel,\\
Pleinlaan 9, 1050 Brussels, Belgium.\\[0.45em]
\href{mailto:brecht.verbeken@vub.be}{brecht.verbeken@vub.be}}}
\date{}
\maketitle

\begin{abstract}
We construct explicit off-diagonal Ramsey graphs from incident
point--hyperplane flags. For each fixed $s$ we obtain a power-law lower
bound in $t$ for $R(s,t)$. The exponent is asymptotic to
$s/(2\log_2 s)$ as $s\to\infty$. Short Frobenius relations improve the
rank bound and give
$R(16,t)\ge\Omega(t^{2.0539221767\ldots})$. We attain the fractional
endpoint by interpolating flags over a field extension and deleting
edges with a polynomial test. For fixed construction parameters, field
initialization, vertex decoding and adjacency take deterministic time
polynomial in the extension degree after finite preprocessing. A
finite-fiber argument gives $R(s,t)\ge t^{(1/2-o(1))s/\log_2 s}$ when
$s\to\infty$ and $s^2(\log_2 s)^2=o(\log_2 t)$.

For $s\ge4$ a direct construction gives
$R(s,t)\ge\Omega_s(t^{(\lfloor s/2\rfloor-1)/\log_2 s})$. A fixed formula
gives its coefficients. For sufficiently large extension degree $h$,
initializing the field, generating the coefficients and deciding adjacency
take deterministic bit time polynomial in $s$ and $h$. This answers
a question of Ihringer and Mattheus about the exponent for fixed $s$.
\end{abstract}

\noindent\textit{Keywords.} Off-diagonal Ramsey numbers, explicit graphs, incident flags, finite fields, polynomial interpolation.\par
\noindent\textit{2020 Mathematics Subject Classification.} Primary 05D10. Secondary 05B25, 14G15, 68R10.\par

\section{Introduction}\label{sec:intro}
The Ramsey number $R(s,t)$ is the least $N$ for which every graph on
$N$ vertices contains an $s$-clique or an independent set of size $t$.
To give an explicit lower bound we construct a graph containing neither
and specify how to decide whether two given vertices are adjacent.
Throughout, $\log$ denotes the logarithm to base two and $\ln$ the
natural logarithm.

Alon and Pudl\'ak \cite[Concluding remarks]{AP} asked for an explicit
construction with $R(s,t)\ge t^{\Omega(s)}$. Ihringer and Mattheus
\cite[Section 5]{IM} asked whether one could obtain $t^{\Omega(s^\varepsilon)}$
for some $\varepsilon>0$. In both questions $s$ is fixed before $t$ tends
to infinity. We start with the rank argument of Ihringer and Mattheus.
Order the vertices, attach an incident point and hyperplane to each,
and join two vertices when the earlier point lies in the later hyperplane.
The rank estimate holds for every choice of flags. This leaves us free to
choose the flags to exclude cliques.

\subsection{The Ramsey bounds}
At the fractional endpoint we delete an edge from each ordered clique.
The dimension of their configuration space bounds the label dimension.

For point vectors with $d$ coordinates, let $\kappa_{s,d}$ denote the
codimension of these configurations. For $3\le d\le s$,
Theorem~\ref{thm:exact-dimension} gives
\[
 \kappa_{s,d}=s(d-1)-\binom d2,\qquad
 \mu_{s,d}=\frac{\kappa_{s,d}}{s-1}.
\]
At the endpoint, $\mu_{s,d}$ is the ratio of label dimension to relative
flag-field degree.

For the Frobenius refinement set
\[
 a=\lceil\mu_{s,d}\rceil,\qquad r=d-1-a,\qquad k=\lceil\log_2s\rceil.
\]
If $r=0$, put $\rho=d$. Otherwise let $\rho$ be the unique root in
$(d-1,d)$ of $x^k-dx^{k-1}+r=0$. Define
\[
 E_{\rm crit}(s)=\max_{3\le d\le s}\frac{\mu_{s,d}}{\log_2 d},
 \qquad
 E_{\rm F}(s)=\max_{3\le d\le s}\frac{\mu_{s,d}}{\log_2\rho}.
\]
In each quotient the denominator comes from the rank bound.
The Frobenius relations replace $d$ by $\rho$ in the second one.

\begin{theorem}\label{thm:main}
For every fixed integer $s\ge3$, deterministic witnessing graphs give
\[
 R(s,t)\ge\Omega_s\bigl(t^{E_{\rm F}(s)}\bigr),\qquad
 E_{\rm F}(s)\ge E_{\rm crit}(s)\ge\frac{s}{2\log_2s}.
\]
Both optimized exponents are asymptotic to $s/(2\log_2s)$ as $s\to\infty$.
For fixed construction parameters, finite deterministic algebraic
preprocessing precedes field initialization, vertex decoding and adjacency.
These three operations have deterministic bit cost polynomial in the
extension degree.
\end{theorem}

Theorem~\ref{thm:critical} and \eqref{eq:Ecrit} give the unrefined
bound. Corollary~\ref{cor:frobenius-exponent} proves the refined bound,
and Proposition~\ref{prop:asymptotic} gives the asymptotics and locates
the maximum for $E_{\rm crit}$.

Consider $s=16$ and $d=13$. If we double the flag-field degree and
delete the first vertex of each clique, \eqref{eq:marked-scalars}
gives exponent $15/(2\log_2 13)=2.026786\ldots$. Deleting the first
edge allows us to reach $\mu_{16,13}=38/5$, hence exponent
$38/(5\log_2 13)=2.05380997\ldots$. The Frobenius relations replace
$13$ by the root $\rho_{16}\in(12,13)$ of $x^4-13x^3+4=0$ and give
\[
 R(16,t)\ge\Omega\bigl(t^{38/(5\log_2\rho_{16})}\bigr)
   =\Omega\bigl(t^{2.0539221767\ldots}\bigr).
\]
The Frobenius relations add about $0.0001122$ to the exponent.

\subsection{The direct construction}
\begin{theorem}\label{thm:direct-main}
For $s\ge4$, put $m=\lfloor s/2\rfloor-1$. Whenever $h$ is at least
the threshold $h_0(s)$ in \eqref{eq:original-parameters} with $d=s$ and
$r=\lceil s/2\rceil$, there is an explicitly specified graph $G$ satisfying
\[
 V(G)=\F_{2^h}^{\,m},\qquad \omega(G)<s,\qquad \alpha(G)\le s^h.
\]
A single deterministic algorithm initializes $G$ and decides adjacency
from two labels in bit time polynomial in $s$ and $h$. We have
$h_0(s)\le2^{O(s^6\log s)}$ and, for fixed $s$,
$R(s,t)\ge\Omega_s(t^{m/\log_2s})$.
\end{theorem}

\begin{proof}
The direct family of Theorem~\ref{thm:direct} copies $m=d-1-r$
coordinates from each label into its point vector. At $d=s$ and
$r=\lceil s/2\rceil$, we have $rs>\binom d2$. The coefficients producing
a clique therefore lie in a proper algebraic set. A degree bound for an
equation of this set lets us choose the coefficients as successive powers
of one field element, without computing the equation. Every label is
retained and the rank bound is $s^h$. Section~\ref{sec:original-exponents}
gives the threshold and bit cost. Choosing $h=\lfloor\log_s(t-1)\rfloor$
gives $2^{mh}\ge 2^{-m}(t-1)^{m/\log_2s}$.
\end{proof}

Since $m/\log_2s\sim s/(2\log_2s)$, this answers the question of
Ihringer and Mattheus for every fixed $0<\varepsilon<1$.
Alon and Pudl\'ak's target is a factor of order $\log s$ larger in
the exponent.

If $E_{\rm dir}(s)$ denotes the optimized exponent of the direct
construction, \eqref{eq:strict-gap} gives
$E_{\rm crit}(s)-E_{\rm dir}(s)=O(1/\log s)$.

\begin{table}[t]
\centering\small
\setlength{\tabcolsep}{5pt}
\begin{tabular}{rccccc}
\toprule
$s$ & $E_{\rm dir}$ & $E_{\rm mark}$ & $E_{\rm crit}$ & $E_{\rm F}$ & Earlier benchmark\\
\midrule
4  & 0.630930 & 0.630930 & 1.051550 & 1.052649 & $5/3$\\
6  & 0.861353 & 1.000000 & 1.205894 & 1.212043 & $2$\\
12 & 1.505150 & 1.666667 & 1.724081 & 1.724306 & $2$\\
15 & 1.734389 & 1.952601 & 1.972525 & 1.972640 & $2$\\
16 & 1.891667 & 1.952601 & 2.053810 & 2.053922 & $2$\\
17 & 1.952601 & 2.101196 & 2.134027 & 2.134033 & $2$\\
18 & 2.047664 & 2.161905 & 2.213284 & 2.213289 & $2$\\
33 & 3.191191 & 3.328234 & 3.334734 & 3.334734 & 2.137243\\
\bottomrule
\end{tabular}
\caption{Comparison of the explicit exponents. $E_{\rm dir}$ and $E_{\rm mark}$ use relative field degree one in \eqref{eq:old-optimized}. $E_{\rm crit}$ and $E_{\rm F}$ use the fractional endpoint. In the last column the decimal at $s=33$ is the limiting exponent of \cite[Theorem 1.3]{IM} as $t\to\infty$ in its uniform formulation. The other entries in that column give fixed-$s$ powers. Decimals are rounded to six places.}
\label{tab:comparison}
\end{table}

\subsection{Comparison with other constructions}
Alon and Pudl\'ak \cite{AP} obtain explicit fixed-$s$ exponents of order
$\sqrt{\log s/\log\log s}$. Theorem~1.3 of Ihringer and Mattheus
\cite{IM} gives exponent
$(1-o(1))\lceil\log_2s\rceil/\log_2(\lceil\log_2s\rceil+1)$,
where the error tends to zero as $t\to\infty$, uniformly for $3\le s\le t$.
Theorem~\ref{thm:direct-main} gives a larger order of growth in $s$
once $t$ is sufficiently large relative to $s$.

At $s=4$, Han, Ihringer and Van Maldeghem \cite{HIV} obtain exponent
$5/3$ from the split Cayley hexagon. Kostochka, Pudl\'ak and R\"odl
\cite{KPR} give exponents $8/5$, $5/3$ and $2$ at $s=4,5,6$,
respectively. A $K_6$-free graph also excludes $K_s$ for $s\ge6$,
so the quadratic bound serves throughout this range. In
Table~\ref{tab:comparison} the endpoint first exceeds it at $16$,
marked deletion over one field at $17$, and direct construction at $18$.

Brada\v{c} \cite{Bradac} proves the non-explicit lower bound
$\Omega_s(t^{s-1}/(\log t)^{2s-4})$ for fixed $s\ge3$.
Table~\ref{tab:comparison} compares graphs given by explicit constructions.

An explicitness claim will specify both the inputs to a query and the
work done in advance. Strong explicitness means that on $N$ indexed
vertices an edge query takes deterministic $\operatorname{poly}(\log N)$
bit time. For the direct graph, $\log_2 N=h(\lfloor s/2\rfloor-1)$,
and even rebuilding the field and the coefficients for a query has
this bound uniformly in $s,h$. In the critical graph the polynomials
are found once for each parameter choice. Their search cost is bounded
in Section~\ref{sec:search-cost}. For fixed $s,d$ the subsequent field
operations, decoding and adjacency have bit cost polynomial in $h$.

When $s$ varies, the direct algorithm gives
$R(s,t)\ge t^{(1/2-o(1))s/\log_2s}$ under
$s^6\log s=o(\log\log t)$, with the uniform local cost of
Theorem~\ref{thm:direct-main}. A different construction extends the
same lower bound on $R(s,t)$ to $s^2(\log s)^2=o(\log t)$.
Corollary~\ref{cor:growing-range} gives an explicit threshold on $t$
in this wider range. Its local complexity statement fixes the
construction parameters.

\subsection{Where the endpoint comes from}
There is a counting argument for the ratio $\kappa_{s,d}/(s-1)$.
Fix $s,d$, write $Q=2^H$, and assign independent uniform chart flags
to $N$ ordered positions. The chart has $Q^{2d-3}$ points. By
Theorem~\ref{thm:exact-dimension} and the affine point count of
Lachaud and Rolland \cite[Theorem 2.1]{LR}, any increasing $s$-tuple
forms a clique with probability at most $c_{s,d}Q^{-\kappa_{s,d}}$.
The expected number of cliques is therefore at most
$c_{s,d}N^sQ^{-\kappa_{s,d}}$. Take $N=\lfloor cQ^{\kappa_{s,d}/(s-1)}\rfloor$
with $c>0$ small enough to make this expectation at most $N/2$.
For some assignment, deleting one vertex from each clique leaves at
least $N/2$ vertices. Their independence number is at most $d^H$
by Lemma~\ref{lem:rank}. This gives the exponent
$\mu_{s,d}/\log_2d$ by alteration. To reproduce it with indices and
local edge queries, we have to specify which flags or edges to keep.

\subsection{From ordered configurations to an edge rule}
Our polynomials interpolate arbitrary output values at any $v$ distinct
labels. Thus a forbidden family of codimension $\kappa$ gives an
incidence of dimension at most $L+mv-\kappa$, where $L$ counts the
coefficients and $m$ counts the label coordinates. With
$m(v-1)<\kappa$, the incidence has dimension below $L+m$, so its
projection to the coefficients and first label satisfies an equation.

At $m(v-1)=\kappa$ there need not be such an equation in one label.
Keep a pair of labels forming an edge of the pattern. Their image
misses the diagonal and has dimension at most the dimension $L+m$
of that diagonal. Were the diagonal contained in the closure, it would
be an irreducible component. A dense set of distinct pairs cannot
have an open part in that component. We can therefore choose an equation
that vanishes on the selected edges and remains nonzero as a polynomial
on the diagonal.

By a triangular rank certificate we mean the following. For every
assignment of chart entries to the ordered labels there is a matrix
$B$ over a field containing the label field, with rank at most $R(H)$,
nonzero diagonal, and $B_{zw}=0$ for every nonedge with $z\prec w$.
The submatrix on an independent set is triangular and invertible.
The bound may depend on the chart-field degree $H$ but is uniform
over all assignments at that degree.

\begin{theorem}[Deleting a required edge]\label{thm:configuration}
Use a chart parametrized polynomially over $\F_2$ by $\A^\nu$, and an
ordered graph rule with triangular rank certificate $R(H)$. For a
finite list of orders of forbidden patterns, let $\Ccal_i$ be a
constructible set defined over $\F_2$ in $(\A^\nu)^{v_i}$,
$v_i\ge2$. Its codimension there over $\overline{\F}_2$ is at least
$\kappa_i$. In order $i$ there is a required edge in fixed positions
$u_i<w_i$. Every occurrence after deletion of any old edges is assumed
to give an original ordered tuple in $\Ccal_i$ with this edge still
present. The list covers all possible orders.

For $b,m\ge1$, use the chart over $\F_{2^{bh}}$ and labels in
$\F_{2^h}^m$, with $\gcd(b,h)=1$. Under restriction of scalars, copy
$m$ base coordinates and interpolate the remaining $b\nu-m$ outputs
with independent blocks of coefficients. Suppose
\[
 m\le b\nu,\qquad m(v_i-1)\le b\kappa_i\quad\text{for all }i.
\]
There are constants $\delta,C$ depending on this model such that for
all sufficiently large $q=2^h$ an indexed set of $q^{m-1}(q-\delta)$
vertices carries a graph avoiding every listed pattern and having
independence number at most $C R(bh)$. After deterministic preprocessing,
decoding is polynomial in $h$ for the fixed model. The same holds for
edge queries whenever it holds for the original ordered rule.
\end{theorem}

In the flag application we take positions $1$ and $2$ in each increasing
clique. Once the labels are known, one evaluation decides whether to
keep their edge. The finite-fiber argument of Appendix~\ref{sec:fibers}
uses a local coloring in place of that evaluation and gives the
threshold used in Corollary~\ref{cor:growing-range}.

Independent polynomial coefficients and estimates for exceptional
configurations also occur in the random algebraic constructions of Bukh
\cite{Bukh} and Bukh and Conlon \cite{BC}. Evasive sets control
intersections with varieties \cite{DKL,LNZ}. Here the forbidden objects
are ordered tuples of distinct labels, and the rank bound applies to
every choice of the coefficients. Earlier low-rank matrix constructions
for Ramsey graphs include the triangle-free graphs of Codenotti, Pudl\'ak
and Resta \cite{CPR}.

\section{The ordered flag graph}\label{sec:rank}
Write $n=d-1$ and let $\Bbbk$ be a field. In projective space a point is
a line in $\Bbbk^d$, represented by any nonzero column $P$ on that line.
A hyperplane is the kernel of a nonzero row $H$. Thus $P$ lies in the
hyperplane precisely when $HP=0$. Such an incident pair is a flag.

Work in the chart in which the first entry of $P$ and the last entry of $H$
are nonzero. We can scale these entries to $1$ independently. Having chosen
the other point coordinates $p$ and the middle entries $\beta$ of the row,
we solve $HP=0$ for the remaining entry of $H$.
\begin{equation}\label{eq:chart}
P(p)=(1,p_1,\ldots,p_n)^T,\qquad
H(p,\beta)=\left(-p_n-\sum_{i=1}^{n-1}\beta_i p_i,\beta_1,\ldots,\beta_{n-1},1\right).
\end{equation}
Every $(p,\beta)\in\Bbbk^n\times\Bbbk^{n-1}$ gives one flag, and the flag
recovers these coordinates. The chart therefore has dimension $2d-3$.

In the original construction we start with a label $z=(z_1,\ldots,z_m)$ and
set $p_i(z)=z_i$ for $1\le i\le m$. We choose the other $r=n-m$ point coordinates
and the $d-2$ slopes as polynomial functions of $z$. The $c=r+d-2$ outputs each
get their own coefficient variables. The point recovers its label from its first
$m$ affine coordinates, so points at distinct labels are distinct. Several
points can have the same hyperplane.

The incidence of $P(z)$ with $H(w)$ can differ from that of $P(w)$ with
$H(z)$. To choose one test for each pair, fix a total order $\prec$ on the
labels and, for $z\prec w$, put
\begin{equation}\label{eq:adjacency}
z\sim w\quad\Longleftrightarrow\quad K(z,w)=H(w)P(z)=0.
\end{equation}
In coordinates we have
\begin{equation}\label{eq:kernel}
K(z,w)=p_n(z)-p_n(w)+\sum_{i=1}^{n-1}\beta_i(w)\bigl(p_i(z)-p_i(w)\bigr).
\end{equation}
Each pair is now tested in just one orientation. Over a field of
characteristic two, the minus signs in the formula are plus signs.

\begin{lemma}\label{lem:rank}
Assign chart flags over $\F_{2^H}$ to a finite ordered set of labels. Let
\begin{equation}\label{eq:feature-space}
\Wcal_H(P)=\spanop\left\{
\prod_{\ell=0}^{H-1}P_{\sigma(\ell)}^{2^\ell}:\
\sigma:\{0,\ldots,H-1\}\longrightarrow\{0,\ldots,d-1\}
\right\}
\end{equation}
where the span is over $\F_{2^H}$ and the products are functions of the
label. The graph \eqref{eq:adjacency} satisfies
$\alpha\le\dim\Wcal_H(P)\le d^H$. The same bounds hold when the flags are
evaluated on any subset of the labels.
\end{lemma}

\begin{proof}
Index the coordinates of $P(z)$ and $H(w)$ from $0$ to $d-1$. Over the field of
order $Q=2^H$, put
\[
B_{z,w}=1-K(z,w)^{Q-1}.
\]
Raising an element of $\F_Q$ to the power $Q-1$ gives $1$ at every nonzero element
and $0$ at zero. Thus $B_{z,w}$ is $1$ when the point at $z$ lies in the hyperplane
at $w$, and $0$ otherwise. In particular, $B_{z,z}=1$.

List an independent set in the order $z_1\prec\cdots\prec z_u$. Every entry above its diagonal comes from a pair $i<j$ with
$K(z_i,z_j)\ne0$, hence $B_{z_i,z_j}=0$. The principal submatrix on the independent
set is lower triangular with diagonal $1$, so
\begin{equation}\label{eq:triangular}
u\le\rank_{\F_Q}B.
\end{equation}
We bound the rank by separating the $z$- and $w$-dependence. Write $Q-1$ as
$1+2+\cdots+2^{H-1}$ and use the additivity of the $2^\ell$-th power in each
factor. This gives
\begin{equation}\label{eq:frobenius-factor}
K(z,w)^{Q-1}=
\prod_{\ell=0}^{H-1}\left(\sum_{i=0}^{d-1}P_i(z)^{2^\ell}H_i(w)^{2^\ell}\right)
=\sum_\sigma U_\sigma(z)V_\sigma(w),
\end{equation}
where
\[
U_\sigma(z)=\prod_{\ell=0}^{H-1}P_{\sigma(\ell)}(z)^{2^\ell},\qquad
V_\sigma(w)=\prod_{\ell=0}^{H-1}H_{\sigma(\ell)}(w)^{2^\ell}.
\]
Each column of $K^{Q-1}$, with the power taken entrywise, belongs to the span of the $U_\sigma$. The constant column does too. Indeed, $P_0=1$ makes
$U_{\sigma_0}=1$ for the constant map $\sigma_0(\ell)=0$, and
\[
B_{z,w}=U_{\sigma_0}(z)\bigl(1-V_{\sigma_0}(w)\bigr)
-\sum_{\sigma\ne\sigma_0}U_\sigma(z)V_\sigma(w).
\]
Thus the column rank is at most $\dim\Wcal_H(P)\le d^H$. The same factorization
works after restricting the evaluations to any subset of labels.
\end{proof}

The bound uses the number of point coordinates and the flag field, independent of
the polynomial degrees. We can therefore choose a polynomial family with enough
free coefficients to control cliques and use the same incidence-rank argument
as in \cite{IM} for every specialization.

\section{Counting ordered clique configurations}\label{sec:geometry}
In a clique, each later hyperplane contains all the preceding points. Let
$\Ccal_{s,d}$ be the space of ordered incident point--hyperplane flags with
\begin{equation}\label{eq:clique-config}
P_i\subseteq H_j\qquad(1\le i<j\le s).
\end{equation}
Here $i<j$ compares positions in the tuple. The geometric conditions allow repeated flags and linearly dependent point representatives. In particular, they include the tuple associated with every clique in the graph.

\subsection{The ranks of the successive spans}\label{sec:exact}
We work over $\overline{\F}_2$. At each of the $s$ positions, we record whether the new point belongs to the span of its predecessors. The last hyperplane bounds the final rank by $d-1$. Write the maximum possible rank and the codimension as follows.
\begin{equation}\label{eq:kappa}
u_* =\min(s,d-1),\qquad
\kappa_{s,d}=s(d-1)-\frac{u_*(2d-u_*-1)}2
=\begin{cases}
s(d-1)-\binom d2,&d\le s+1,\\
\binom s2,&d\ge s+1.
\end{cases}
\end{equation}
At $d=s+1$ both formulas give $\binom{s}{2}$.

\begin{theorem}\label{thm:exact-dimension}
The space $\Ccal_{s,d}$ of ordered flags satisfying \eqref{eq:clique-config} has
dimension
\[
\dim\Ccal_{s,d}=s(d-2)+\frac{u_*(2d-u_*-1)}2.
\]
Restricting each flag to the affine chart \eqref{eq:chart} preserves this dimension. Thus the ordered clique conditions have codimension $\kappa_{s,d}$ in $\A^{s(2d-3)}$.
\end{theorem}

\begin{proof}
Put $R_0=0$ and $R_j=\spanop(P_1,\ldots,P_j)$. The ordered incidences place
$P_1,\ldots,P_{j-1}$ in $H_j$ and self-incidence places $P_j$ there as well.
Thus $R_j\subseteq H_j$ and each $R_j$ is proper. Fix a sequence of ranks
$\rho_j=\dim R_j$ ending at $u$. It starts at $1$ and increases by at most
one at each step.

Suppose the flags through position $j-1$ are known and the new rank is $v$.
If the rank stays the same, the point ranges over $\PP(R_{j-1})$, of dimension
$v-1$. We then choose a hyperplane containing $R_j$ from the projective dual
of $\Bbbk^d/R_j$, of dimension $d-v-1$. The total for this step is $d-2$.

If the rank increases to $v$, the point ranges over the open set
$\PP^{d-1}\setminus\PP(R_{j-1})$ of dimension $d-1$. The hyperplane still
contributes $d-v-1$. The step now contributes $d-2$ plus $d-v$, including for
the first point, which takes the rank from $0$ to $1$. Each of $1,\ldots,u$
is reached by one such increase. The dimension for this rank sequence is
\begin{equation}\label{eq:rank-stratum}
s(d-2)+\sum_{v=1}^u(d-v)=s(d-2)+\frac{u(2d-u-1)}2.
\end{equation}
To justify adding these dimensions, start at a point and induct on $j$.
On a stratum of fixed ranks the spaces $R_i$ are vector subbundles, as is
seen by choosing bases on charts where suitable minors are nonzero.
The next point is chosen in a projective bundle or in a nonempty open
subset of one. After the point is chosen, $H_j$ belongs to the projective
bundle of lines in $(\Bbbk^d/R_j)^*$. Over an irreducible base both choices
give an irreducible space, with dimension equal to that of the base plus
the fiber dimension. This gives \eqref{eq:rank-stratum} by induction. Since \eqref{eq:rank-stratum} increases with $u\le u_*$, the maximum is
at $u_*$. We can attain it by taking the first $u_*$ points independent and
choosing all other points in their span.

To attain the dimension in the affine chart, choose a final span that excludes the
last coordinate vector $e_{d-1}$ and on which the first coordinate is a nonzero
linear functional. Choose the points in this span with the required ranks and
nonzero first coordinate. For each $j$, the class of $e_{d-1}$ in
$\Bbbk^d/R_j$ is nonzero. Hence there is a functional killing $R_j$ and nonzero
on $e_{d-1}$. Scaling it makes its last coefficient $1$. In the bundle description,
each of these nonvanishing conditions selects a nonempty open part of the space
of choices, so the dimension remains \eqref{eq:rank-stratum} at $u=u_*$.
Subtracting it from $s(2d-3)$ yields \eqref{eq:kappa}.
\end{proof}

For example, with three flags in $d=3$ the rank path $1,2,2$ contributes
$3+2+1=6$ dimensions. If the points all coincide, the path $1,1,1$ contributes
$3+1+1=5$. In the former case the second and third hyperplanes are both forced
to be the common span.

For $d\le s+1$, the dimension agrees with the complete-flag bound of
Lemma~\ref{lem:complete-flag-bound}. In the original model with $m\ge1$, choose
the final span so the projection to the first $m$ affine point coordinates is
nonconstant. A generic choice of points of maximal rank then avoids all finitely
many pairwise label coincidences. Thus the distinct-label condition leaves an open subset of a stratum of
maximum dimension. For the construction it would suffice to know the upper
bound on the dimension.

\section{Interpolating at every distinct tuple}\label{sec:interpolation}
When we count coefficient choices for a clique, we have to interpolate at its
actual labels. We want one family of features that works at every tuple of
$s$ distinct labels. To find it, view each $z\in\Bbbk^m$ as the coefficient vector
of a polynomial $\ell_T(z)$ in a formal variable $T$. Take the coefficients
of its first $s$ powers.
\begin{equation}\label{eq:compressed-features}
\ell_T(z)=z_1+z_2T+\cdots+z_mT^{m-1},\qquad
\phi_{j,e}(z)=[T^e]\ell_T(z)^j.
\end{equation}
The notation $[T^e]$ extracts the coefficient of $T^e$. Use $0\le j<s$ and
$0\le e\le(m-1)j$. The feature $\phi_{j,e}$ has degree at most $j$ in $z$.
We use
\begin{equation}\label{eq:M}
M(m,s)=\sum_{j=0}^{s-1}\bigl((m-1)j+1\bigr)=s+(m-1)\binom s2
\end{equation}
coefficient slots. Some slots may give the same function or the zero function.
For instance, with $m=2$ in characteristic two,
$(z_1+z_2T)^2=z_1^2+z_2^2T^2$, so $\phi_{2,1}=0$ when that slot is used.
The next lemma concerns the rank after evaluation.

\begin{lemma}[Universal interpolation]\label{lem:interpolation}
For any $s$ distinct labels in $\Bbbk^m$, evaluation of the
$M(m,s)$ features at those labels maps their coefficients onto $\Bbbk^s$.
Equivalently, the matrix
\[
E(z_1,\ldots,z_s)=\bigl(\phi_{j,e}(z_i)\bigr)_{i,(j,e)}
\]
has row rank $s$ over $\Bbbk$.
\end{lemma}

\begin{proof}
Over an algebraic closure, we can choose $\tau$ so that the scalars
$a_i=\ell_\tau(z_i)$ are pairwise distinct. Indeed, each difference
$\ell_T(z_i)-\ell_T(z_{i'})$ for $i\ne i'$ has a nonzero coefficient and has
only finitely many roots. The choice of $\tau$ gives access to ordinary
one-variable interpolation, since for $j<s$,
\[
\ell_\tau(z)^j=\sum_{e=0}^{(m-1)j}\tau^e\phi_{j,e}(z).
\]
The evaluation vectors of $1,\ell_\tau,\ldots,\ell_\tau^{s-1}$ are linear
combinations of the columns of $E$. They form a Vandermonde matrix with nonzero
determinant $\prod_{i<i'}(a_{i'}-a_i)$. Hence some $s\times s$ minor of $E$ is
nonzero over the algebraic closure. That minor belongs to the field of the
labels, so the rank there is also $s$.
\end{proof}

Order the slots first by $j$ and then by $e$, and assign a separate block
to each output coordinate. For the original chart, take the outputs in the order
$p_{m+1},\ldots,p_n,\beta_1,\ldots,\beta_{n-1}$. This orders the
$L=cM(m,s)$ coefficients $A=(A_0,\ldots,A_{L-1})$. The parameter space counts
coefficient choices, so two choices giving the same output functions still
define separate parameter tuples.

\begin{lemma}\label{lem:universal-dimension}
Let $\Ccal\subseteq(\A^\nu)^s$ be constructible with
$\dim\Ccal\le\nu s-\kappa$. Copy the same $m$ coordinates in each chart entry to form its label and
consider configurations whose $s$ labels are distinct. Allow each remaining
coordinate its own block of $M(m,s)$ coefficients in
\eqref{eq:compressed-features}, so $c=\nu-m$ and $L=cM(m,s)$.
Let $\Ical$ consist of the configurations together with coefficient choices
that interpolate their output coordinates. Then
\begin{equation}\label{eq:universal-dimension}
\dim\Ical\le L+ms-\kappa.
\end{equation}
\end{lemma}

\begin{proof}
Denote the distinct-label locus by $\Ccal^\circ$. If it is empty, the
assertion follows at once. Otherwise let $U$ be the part of $\Ccal^\circ$
on which a specified $s\times s$ evaluation minor is nonzero. For each
output choose the $M-s$ coefficients outside the minor freely. Subtract
their contribution from the $s$ output values and solve with the chosen
minor for the remaining coefficients. Its inverse is regular on $U$.

We have thus identified the incidence above $U$ with
$U\times\A^{c(M-s)}$. By Lemma~\ref{lem:interpolation} these $U$ cover
$\Ccal^\circ$, so the projection forgetting all coefficients is
surjective. The finite cover gives
\[
\dim\Ical=\dim\Ccal^\circ+L-cs\le L+ms-\kappa.
\]
\end{proof}

For the original flag model, write $\Ccal_{s,d}^{\circ}$ for the configurations
in the chart with distinct first $m$ affine point-coordinate labels. A point
of $\Ical$ gives coefficients that realize each flag at its own label.
Those coefficients and labels determine the output values, so the incidence
can equally be recorded in coefficient--label space. Use $m=d-1-r$ and the bound
$\kappa_{s,d}\ge s(d-1)-\binom d2$ in Lemma~\ref{lem:universal-dimension} to get
\begin{equation}\label{eq:original-dimension}
\dim\Ical\le L+\binom d2-rs.
\end{equation}
Projecting $\Ical$ onto the coefficients gives a proper image closure as soon as $rs>\binom d2$. Retaining one label increases the target dimension to $L+m$. For this projection the image closure is proper when $\binom d2-rs<m$, or equivalently
$r(s-1)>\binom{d-1}2$. These are the strict inequalities used in
Section~\ref{sec:original-statements}.

\section{Finding and specializing an equation}\label{sec:algebra}
We seek an equation for the exceptional coefficient choices, or for these
choices together with one or two labels. Its degree will tell us how
large the field has to be for specialization. To include components of
every dimension, use the cumulative degree of a reduced affine set $X$,
\[
\cdeg X=\sum_{Z\text{ an irreducible component of }X}\deg Z,
\qquad \cdeg\varnothing=0.
\]
For an irreducible affine set, its degree is that of its projective closure.
A generic complementary affine linear section has that many points. In
particular, $\cdeg X$ counts the geometric points of a finite reduced set.

\Needspace{4\baselineskip}
\begin{lemma}\label{lem:degree}
The following statements hold over an algebraically closed field.
\begin{enumerate}
\item If $f_1,\ldots,f_u$ have positive degree, then
\begin{equation}\label{eq:bezout}
\cdeg\bigl(X\cap V(f_1,\ldots,f_u)\bigr)\le\cdeg X\prod_{i=1}^u\deg f_i.
\end{equation}
\item If $\pi$ is linear, then $\cdeg\overline{\pi(X)}\le\cdeg X$.
\item A proper closed set $Y\subsetneq\A^N$ is the common zero set of
polynomials of degree at most $\cdeg Y$. In particular, for any $x\notin Y$
one such polynomial is nonzero at $x$. For $Y=\varnothing$, take $1$.
\item Over $\overline{\F}_2$, suppose $Y$ is invariant under coordinatewise
Frobenius and its inverse. The polynomials of degree at most $t$ vanishing
on $Y$ have a basis with coefficients in $\F_2$. If some such polynomial
is nonzero at $x$, a basis element is nonzero there as well.
\end{enumerate}
\end{lemma}

\begin{proof}
For \eqref{eq:bezout}, intersect one irreducible component with one hypersurface
at a time \cite[Theorem 1.1]{LR}. A component contained in the hypersurface
keeps its degree, which is bounded by the product since the hypersurface has
positive degree. For a component cut by the hypersurface, projective B\'ezout
bounds the sum of degrees of the intersection components by the product.
Sum over components, allowing pieces counted more than once, and iterate.
A zero equation can be omitted and a nonzero constant equation makes the
intersection empty.

For the projection, first take $X$ irreducible of dimension $e$ and write
$Y=\overline{\pi(X)}$ with $f=\dim Y$. Constructibility and the fiber-dimension
theorem \cite[Chapter I]{Hartshorne} give a nonempty open subset of $Y$ in the
image where all fibers have dimension $e-f$. Intersect $Y$ with $f$ generic
affine hyperplanes, obtaining $\deg Y$ distinct points in this open subset.
Pull the $f$ equations back to $X$. Add $e-f$ generic hyperplanes in the source
to cut each of the finitely many selected fibers to a finite nonempty set.
We have found at least $\deg Y$ points in a finite linear section of $X$.
By \eqref{eq:bezout} there are at most $\deg X$. For reducible $X$, the
components of the image closure are among the closures of the component
images. Summing gives the desired inequality.

Heintz's set-theoretic degree bound \cite[Proposition~3]{Heintz}
gives defining equations for $Y$ of degrees at most the sum of the
degrees of its components. A point outside $Y$ fails at least one of
these equations. The polynomial $1$ covers the empty case.

For descent, fix a degree bound and identify a polynomial with its vector of
coefficients in the monomials up to that degree. The polynomials vanishing on
$Y$ give a linear subspace $V$ of this vector space. For $f=\sum c_\nu X^\nu$ in
$V$, the polynomial $f^{(2)}=\sum c_\nu^2X^\nu$ also belongs to $V$. Indeed, if
$y\in Y$, its coordinatewise square root $x$ belongs to $Y$ and
$f^{(2)}(y)=f(x)^2=0$. The inverse Frobenius preserves $V$ by the same reasoning.

Write a basis of $V$ in reduced row-echelon form. Squaring every entry
preserves the row space and the pivot conditions. Uniqueness of this form
forces each entry to be fixed by squaring, so each row has coefficients in
$\F_2$. If every row vanished at $x$, every polynomial in $V$ would
vanish there by linearity of evaluation. Thus a row with binary coefficients
separates $x$ whenever an element of $V$ does.
\end{proof}

A zero set given by binary polynomials is stable under both Frobenius and
its inverse. Projection onto coordinates commutes with these operations,
and taking the closure preserves stability because they are homeomorphisms.
We can therefore apply descent to the whole image closure, even when its
components are permuted.

\Needspace{6\baselineskip}
\subsection{Bounding degree after substitution}\label{sec:degree-model}

\begin{proposition}\label{prop:substituted-degree}
In a fixed degree-$b$ restriction-of-scalars model, suppose that each point and
slope coordinate after substitution has total degree at most $e$ in coefficients
and labels. Let $X$ be the closure of the clique incidence after imposing the
open distinctness conditions. Then
\begin{equation}\label{eq:substituted-degree}
\cdeg X\le\mathcal D:=(2e)^{b\binom s2}.
\end{equation}
For a single-stage compressed family, one may take $e=s$. For the two-stage family of Section~\ref{sec:two-stage}, one may take $e=1+s(s-1)$.
\end{proposition}

\begin{proof}
After substitution, the incidence expression \eqref{eq:kernel} has degree
at most $2e$. Multiplication in the degree-$b$ algebra is bilinear in base
coordinates with fixed structure constants. Thus each clique pair gives $b$
scalar equations of degree at most $2e$. Applying \eqref{eq:bezout} to these
equations bounds the cumulative degree of their common zero set $V$ by
$(2e)^{b\binom s2}$.

Let $O$ be the open set imposing distinctness. Each component of $V$ that meets
$O$ meets it in a dense subset, while the other components contribute no points
to $V\cap O$. The closure of $V\cap O$ is the union of the components meeting
$O$, so it has cumulative degree at most that of $V$.

In the single-stage family, each feature has degree at most $s-1$ in the
labels and its multiplier is one coefficient, giving degree at most $s$.
In the two-stage construction the input coordinates for the second
interpolation have degree at most $s$. A feature of degree $j\le s-1$
in those coordinates has degree at most $sj$ after substitution. Its
coefficient adds at most one, giving $1+s(s-1)$. The change to base-field
coordinates preserves these bounds because the structure constants for
multiplication are fixed.
\end{proof}

For a polynomial-system computation, it is useful to express the open
conditions by equations in more variables. To express $z_i\ne z_j$, introduce
$u_{ij,1},\ldots,u_{ij,m}$ and impose
\begin{equation}\label{eq:inverse-witness}
\sum_{\ell=1}^m u_{ij,\ell}(z_{i,\ell}-z_{j,\ell})=1.
\end{equation}
For a fixed pair of labels, this linear equation in the $u$'s has a
solution exactly when a coordinate difference is nonzero. To test whether a candidate $Q$ vanishes on a projected image,
adjoin $vQ-1=0$. A point where $Q$ is nonzero extends to a solution by
setting $v=1/Q$, and every solution projects to such a point. Thus $Q$ vanishes on the image and
its closure precisely when the enlarged ideal contains $1$, by the
Nullstellensatz. A Gr\"obner basis over $\F_2$ decides this in finite time,
since extending the field preserves whether the ideal is the unit ideal.
We use the extra variables in the computation. The dimension and degree
estimates above concern the incidence in its original coordinates.

\subsection{Specializing one coefficient at a time}\label{sec:adaptive}

\begin{lemma}\label{lem:adaptive}
Let $R\in\F_q[A_0,\ldots,A_{L-1}]$ be nonzero of total degree at most $\delta$,
with $q>\delta$. Choose any list of $\delta+1$ distinct field elements. A deterministic
procedure can select each coordinate of $a$ from this list, with
$R(a)\ne0$, by keeping the partially evaluated polynomial nonzero.
\end{lemma}

\begin{proof}
Suppose the coordinates before $A_i$ have been fixed and the remaining
polynomial is nonzero. Expand it in the variables after $A_i$. At least one
coefficient $g(A_i)$ is a nonzero polynomial of degree at most $\delta$.
If the remaining polynomial becomes zero upon assigning a value to $A_i$,
then $g$ vanishes at that value. There are at most $\delta$ such roots.
Testing the first $\delta+1$ distinct field elements finds a value for which
the remaining polynomial is nonzero. At the last step this polynomial is the
constant $R(a)$.
\end{proof}

For $Q(A,Z)=\sum_\gamma R_\gamma(A)Z^\gamma$, apply the lemma to the first
nonzero $R_\gamma$. It guarantees $Q(a,Z)\not\equiv0$. For fixed degrees and
dimensions, comparing and specializing the coefficients takes bit time
polynomial in $h$ when $q=2^h$. For a computed polynomial of degree at most $\delta$, the field
condition for this step is $q>\delta$. Taking $q>2\delta$ further
ensures that the decoder can index at least half the labels.
Section~\ref{sec:original-statements} gives the larger sufficient threshold
when the coefficients are chosen before the exceptional polynomial is computed.

\section{Constructions from strict dimension inequalities}\label{sec:original-statements}
The strict bounds in \eqref{eq:original-dimension} allow a simpler choice
of what to delete. When the image in coefficient space has proper
closure, a degree bound suffices to choose coefficients while keeping
every label. When remembering the first label gives an image with
proper closure, a computed polynomial deletes a set of vertices.

For one copy of the flag chart, put
\begin{align}\label{eq:original-parameters}
&s\ge3,\quad d\ge3,\quad 1\le r\le d-2,\quad m=d-1-r,\quad D_0=\binom d2,\notag\\
&c=r+d-2,\quad M=s+(m-1)\binom s2,\quad L=cM,\notag\\
&a_0=s(s+c-1),\quad \Delta_0=s^{a_0},\quad B_0=\Delta_0+1,\quad h_0=B_0^L.
\end{align}
Each of the $c$ outputs has a block of $M$ coefficients. We count the full block even if two features agree or one vanishes. The closed incidence model below uses $a_0$ equations and has cumulative degree at most $\Delta_0$.

\begin{theorem}[Direct construction]\label{thm:direct}
If $rs>D_0$, then for every integer $h\ge h_0$ there is a deterministically specified graph on $(\F_{2^h})^m$ with
\[
|V|=2^{hm},\qquad \omega<s,\qquad \alpha\le d^h.
\]
Its coefficients are given by \eqref{eq:kronecker}, and the clique and independence bounds hold in every total order. With the binary labels ordered lexicographically, the field, coefficients and adjacency are computable in deterministic polynomial time in $h$ for fixed $s,d,r$. For $s\ge4$, $d=s$ and $r=\lceil s/2\rceil$, the polynomial bit bound holds jointly in $s$ and $h$.
\end{theorem}

\begin{theorem}[Marked deletion]\label{thm:marked}
If $r(s-1)>\binom{d-1}2$, then for every $h\ge h_0$, with $q=2^h$, there is a deterministically specified graph with
\[
|V|=q^{m-1}(q-\Delta_0),\qquad \omega<s,\qquad \alpha\le d^h.
\]
After deterministic preprocessing depending on $s,d,r$, the vertices have product indices. For fixed $s,d,r$ we can decode an index or decide adjacency in bit time polynomial in $h$.
\end{theorem}

On each of the $q^{m-1}$ lines the decoder retains $q-\Delta_0$ labels. Since $h\ge h_0>\Delta_0$, this leaves more than half the label space. In both graphs the point coordinate $1$ includes the constant feature in the rank bound $d^h$.

\begin{corollary}
\label{cor:original-ramsey}
For fixed $s,d,r$ satisfying the assumptions of Theorem~\ref{thm:direct} or Theorem~\ref{thm:marked}, the corresponding graphs give
\[
 R(s,t)\ge\Omega_{s,d,r}\bigl(t^{(d-1-r)/\log_2d}\bigr).
\]
At $d=s$, the choices $r=\lceil s/2\rceil$ in the direct construction ($s\ge4$) and $r=\lfloor s/2\rfloor$ in marked deletion ($s\ge3$) give the respective exponents
\[
 \frac{\lfloor s/2\rfloor-1}{\log_2s},\qquad
 \frac{\lceil s/2\rceil-1}{\log_2s}.
\]
\end{corollary}

The positive constant and the threshold for $t$ in
$\Omega_{s,d,r}$ may depend on $s,d,r$.

\subsection{Equations for the exceptional sets}\label{sec:strict}
\subsubsection{A closed witness model}
Take the parameters in \eqref{eq:original-parameters}. Include the
$L$ coefficients and the $2d-3$ chart coordinates for each of the
$s$ flags. Let $z_j$ be the first $m$ affine point coordinates at
position $j$. We want a closed set $W$ whose projection gives the
incidence with distinct labels. Use $\binom s2$ clique equations of
degree at most two and $\binom s2$ equations \eqref{eq:inverse-witness}
of degree two, adding a block of witness variables for every pair.
For each of the $c$ output coordinates at each position, equate the
coordinate to its interpolating polynomial. Each of these $cs$
equations has total degree at most $s$ including its coefficients.
The count is
\[
2\binom s2+cs=s(s+c-1)=a_0
\]
and \eqref{eq:bezout} bounds $\cdeg W$ by $s^{a_0}=\Delta_0$.

Over the algebraic closure, forgetting the witnesses maps $W$ onto
the distinct-label incidence $\Ical$. In particular, the projections
of $W$ and $\Ical$ onto $A$ coincide, and so do their projections onto
$(A,z_1)$. Under the respective strict inequalities in
\eqref{eq:original-dimension}, the closures of these images are
proper. The projection bound and descent in Lemma~\ref{lem:degree}
then give
\begin{align}\label{eq:exceptional-original}
&0\ne Q_A(A)\in\F_2[A],\quad\deg Q_A\le\Delta_0
&&\text{if }rs>\binom d2,\notag\\
&0\ne Q(A,Z)\in\F_2[A,Z],\quad\deg Q\le\Delta_0
&&\text{if }r(s-1)>\binom{d-1}2.
\end{align}
The polynomial $Q_A$ vanishes at the coefficients of every
ordered clique, and $Q$ vanishes at those coefficients together
with its first label. We used the dimension of $\Ical$ for
properness and the degree of $W$ to bound each polynomial.

\subsection{Choosing coefficients using only the degree bound}
A common evaluation point for all nonzero binary polynomials of degree at most $\Delta_0$ will suffice. The following Kronecker substitution encodes each coefficient monomial by a different power of one element. Choose an irreducible $f_h(T)\in\F_2[T]$ of
degree $h\ge h_0$ and put
\[
\F_q=\F_2[T]/(f_h(T)),\qquad\theta=T\pmod{f_h(T)}.
\]
The element $\theta$ has minimal-polynomial degree $h$. In the
coefficient order of Section~\ref{sec:interpolation}, take
\begin{equation}\label{eq:kronecker}
A_j=\theta^{B_0^j}\qquad(0\le j<L).
\end{equation}

\begin{lemma}\label{lem:kronecker}
Every nonzero polynomial in $\F_2[A_0,\ldots,A_{L-1}]$ of total degree
at most $\Delta_0$ has nonzero value at \eqref{eq:kronecker}.
\end{lemma}

\begin{proof}
Let $\widehat Q(T)$ be the polynomial obtained by substituting $T^{B_0^j}$ for $A_j$. The monomial with exponent vector $(e_0,\ldots,e_{L-1})$ becomes $T^n$ with $n=\sum_j e_jB_0^j$. The entries satisfy $0\le e_j\le\Delta_0<B_0$ and can be recovered as the digits of $n$. Thus $Q\ne0$ gives $\widehat Q\ne0$. Moreover,
\[
 n\le B_0^{L-1}\sum_j e_j\le B_0^{L-1}\Delta_0<h_0\le h.
\]
A nonzero polynomial of degree less than $h$ cannot vanish at $\theta$, whose minimal polynomial has degree $h$. Its value here is $Q$ at \eqref{eq:kronecker}.
\end{proof}

\begin{proof}[Proof of Theorem~\ref{thm:direct}]
If the chosen graph contained an $s$-clique, list its vertices in increasing order. Their flag tuple and its coefficients would belong to the incidence whose projection is annihilated by $Q_A$. Lemma~\ref{lem:kronecker} gives $Q_A(A)\ne0$ at \eqref{eq:kronecker}. We may therefore retain every label. The point coordinates distinguish all $q^m$ of them, and Lemma~\ref{lem:rank} gives $\alpha\le d^h$.

Use the field representation produced by Shoup's deterministic algorithm in characteristic two \cite{Shoup}. It is constructed in time polynomial in $h$ and permits field arithmetic in polynomial bit time in $h$. We concatenate the $m$ binary coordinates to index a vertex and compare these strings to order the endpoints of an edge.

There are two computations to count. To obtain all $A_j$ we start at $\theta$ and use $A_{j+1}=A_j^{B_0}$. Repeated squaring costs $O(L\log B_0)$ field operations. If $d=s$ and $r=\lceil s/2\rceil$, then $L=O(s^4)$ and the integer $B_0$ is obtained by powering and addition on $O(s^2\log s)$ bits. To obtain the outputs at $z$, form the powers $1,\ell_T(z),\ldots,\ell_T(z)^{s-1}$ by successive multiplications of coefficient arrays. Each multiplication uses at most $m(1+(s-1)(m-1))$ field multiplications and as many additions. Reading the entries and applying the $c$ coefficient blocks gives all the coordinates for \eqref{eq:kernel}. Both costs are polynomial in $s,h$ on the stated family and polynomial in $h$ whenever $s,d,r$ are held fixed.
\end{proof}

Only the bound on $\deg Q_A$ enters the coefficient formula. The witness equations yielding that bound depend on $s,d,r$ alone.

\subsection{Deleting the first label}
For marked deletion we need the equation itself. Take the system defining $W$ and append $vQ(A,z_1)-1=0$. Its ideal contains $1$ precisely when $Q$ vanishes on the projection to $(A,z_1)$. Apply the test of Section~\ref{sec:degree-model} to each nonzero binary polynomial of degree at most $\Delta_0$ in a fixed order. A candidate passes by \eqref{eq:exceptional-original}. This determines $Q$ before $h$ is chosen.

Expand the accepted equation as $Q(A,Z)=\sum_\gamma R_\gamma(A)Z^\gamma$. Choose one nonzero $R_\gamma$. Its degree is at most $\Delta_0$, so at the coefficients \eqref{eq:kronecker} its value is nonzero. We obtain a nonzero polynomial $F(Z)=Q(A,Z)$ of degree at most $\Delta_0$. An ordered clique gives a zero of $F$ at its first label. Hence the induced graph on $F(z)\ne0$ has no $s$-clique, and restriction preserves the bound on its independence number.

For a nonzero polynomial of degree $\delta$ in $m$ variables,
the number of zeros is at most $\delta q^{m-1}$. In one variable
this is the root bound. For the induction, let $e$ be the degree
in the last variable. The leading coefficient in the other
variables has degree at most $\delta-e$. It vanishes on at most
$(\delta-e)q^{m-2}$ choices by induction, and those choices contribute
at most $q$ zeros each. The remaining choices contribute at most
$e$ roots each. The total is at most
$(\delta-e)q^{m-1}+eq^{m-1}$. A nonzero constant contributes no
zeros. Thus at least $q^m-\Delta_0q^{m-1}$ labels survive.
Lemma~\ref{lem:decoder} selects exactly that many in a decodable
product set and gives the local cost in Theorem~\ref{thm:marked}.

\begin{corollary}[Specializing a computed equation]\label{cor:adaptive-marked}
If $2^h>2\Delta_0$, computing the equation $Q(A,Z)$ and applying Lemma~\ref{lem:adaptive} gives the graph and local complexity of Theorem~\ref{thm:marked}. Thus it suffices to take $h\ge\lceil\log_2(2\Delta_0+1)\rceil$. With $Q_A$ computed in advance the same field-size condition suffices in Theorem~\ref{thm:direct}.
The graph parameters and the local bit bound for fixed $s,d,r$ hold
after this preprocessing. The uniform construction in
Theorem~\ref{thm:direct-main} uses the threshold stated there.
\end{corollary}

\begin{proof}
For $Q_A$ it is enough to find $a$ with $Q_A(a)\ne0$. For $Q(A,Z)$ choose a nonzero $R_\gamma$ and seek $R_\gamma(a)\ne0$, so that the specialized equation in $Z$ survives. Lemma~\ref{lem:adaptive} makes both choices when $q>\Delta_0$. Under $q>2\Delta_0$ the marked construction has the stated decoded set and the direct construction still has every label.
\end{proof}

\subsection{Optimizing the strict bounds}\label{sec:original-exponents}
Given $t\ge d^{h_0}+1$, choose $h=\lfloor\log_d(t-1)\rfloor$. This gives $h\ge h_0$ and $d^h\le t-1$. The direct graph has
\[
 2^{hm}\ge 2^{-m}(t-1)^{m/\log_2d}
\]
vertices. Marked deletion keeps at least half the labels by $2^h>2\Delta_0$. In either case we obtain the exponent in Corollary~\ref{cor:original-ramsey}.

Put $d=s$. The direct inequality is $r>(s-1)/2$, so its least integer solution is $\lceil s/2\rceil$. For $s\ge4$ this leaves $m=\lfloor s/2\rfloor-1\ge1$. With marked deletion the condition becomes $r>(s-2)/2$. Its least integer solution $\lfloor s/2\rfloor$ leaves $m=\lceil s/2\rceil-1\ge1$ for $s\ge3$.

For a fixed $d$, both bounds improve with the number $m=d-1-r$
of copied coordinates. Thus we use the least possible integer
$r$ and optimize over $d$.
\begin{align}\label{eq:old-optimized}
E_{\rm dir}(s)&=\max_d\frac{d-2-\lfloor d(d-1)/(2s)\rfloor}{\log_2d},\notag\\
E_{\rm mark}(s)&=\max_d\frac{d-2-\lfloor(d-1)(d-2)/(2(s-1))\rfloor}{\log_2d}.
\end{align}
The maxima use only $d$ for which the corresponding $r$ satisfies
$1\le r\le d-2$. This requires $d<2s$ for the direct construction
and $d<2s-1$ for marked deletion. At $s=3$ there is no direct
choice in this range.

For either strict construction we have
\[
 m<\frac{(d-1)(s-d/2)}{s-1}\le \frac{s}{2}.
\]
Indeed, every direct choice satisfies the strict marked inequality, and the second inequality is equivalent to $(d-s)(d-s-1)\ge0$. As $d\ge m+2$ and $x/\log_2(x+2)$ increases for $x>0$, the exponent is at most $(s/2)/\log_2(s/2+2)$. Together with $d=s$ this yields $E_{\rm dir},E_{\rm mark}\sim s/(2\log_2s)$.

For an additive comparison let $d_*$ attain the critical exponent, so $d_*\sim s$ by Proposition~\ref{prop:asymptotic}. Put $\kappa=\kappa_{s,d_*}$ and $\mu=\kappa/(s-1)\le s/2$. The direct condition at this $d_*$ is $m<\kappa/s$. For sufficiently large $s$ the largest such integer is a valid, positive label count and is at least $\kappa/s-1$. The difference from $\mu$ is at most $1+\kappa/[s(s-1)]\le3/2$. For the reverse comparison, every direct exponent with $d\le s$ is below $\mu_{s,d}/\log_2d$. If $d>s$, the bound $m<s/2$ puts it below the critical exponent at $d=s$. We conclude that
\begin{equation}\label{eq:strict-gap}
0\le E_{\rm crit}(s)-E_{\rm dir}(s)\le \frac{3}{2\log_2d_*}
=O\bigl(1/\log s\bigr).
\end{equation}

It remains to put the field degree and the query cost on the same scale. In the direct choice $d=s$ and $r=\lceil s/2\rceil$, we have
\[
 M=O(s^3),\quad L=O(s^4),\quad a_0=O(s^2),\quad
 \log_2 B_0=O(s^2\log s).
\]
Multiplying $L$ by $\log_2 B_0$ gives $\log_2 h_0=O(s^6\log s)$. For $h=\lfloor\log_s(t-1)\rfloor$ it suffices to take $t\ge s^{h_0}+1$, at most $2^{2^{O(s^6\log s)}}$. This condition holds eventually along every sequence with $s\to\infty$ and $s^6\log s=o(\log\log t)$. The bound $N=2^{mh}\ge 2^{-m}(t-1)^{m/\log_2s}$ with $m=\lfloor s/2\rfloor-1$ gives the exponent $(1/2-o(1))s/\log_2s$. A label has $\log_2 N=mh$ bits, and both $h$ and $s$ are $O(mh)$ on this family. The deterministic polynomial in $s,h$ proved above consequently bounds the local query cost by a polynomial in the label length, with an absolute exponent. Appendix~\ref{sec:fibers} obtains a larger range for $s$ by allowing preprocessing depending on $s$.

\section{Deleting edges at the critical dimension}\label{sec:filter}
If the coefficients and one label have proper image closure, we can delete
a vertex of every forbidden tuple. At equality the one-label image may be
dense. We therefore remember the labels at the ends of an edge. A
polynomial vanishing on these ordered pairs will delete that edge. Its
values at equal labels will decide whether the same triangular rank
argument still bounds independent sets.

\begin{lemma}\label{lem:filter-rank}
Let the entries of $B$ and the values of the functions below lie in a
common field. Assume $B_{z,z}\ne0$ and
$B_{z,w}=0$ whenever $z\prec w$ is a nonedge. Assume also that
\[
F(Z,W)=\sum_{i=1}^R f_i(Z)g_i(W).
\]
Retain the vertices with $F(z,z)\ne0$. Among these vertices keep an old edge
$z\prec w$ precisely when $F(z,w)\ne0$. The graph so obtained has
$\alpha\le R\rank B$. If the labels have $m$ coordinates and $F$ has degree
at most $\delta$ in $W$, one can use $R\le\binom{m+\delta}{m}$.
\end{lemma}

\begin{proof}
Put $M=B\circ(F(z,w))_{z,w}$, where $\circ$ is entrywise multiplication.
On an independent set of the new graph each pair $z\prec w$ is either an
old nonedge or a deleted edge. In the first case $B_{z,w}=0$, and in the
second $F(z,w)=0$. The principal submatrix of $M$ therefore has zero entries
above the diagonal. Its diagonal entries $B_{z,z}F(z,z)$ are nonzero by our
choice of vertices. Its size is at most $\rank M$.

If $B=\sum_{j=1}^{R_B}u_jv_j^T$ with $R_B=\rank B$, then
\[
M=\sum_{j=1}^{R_B}\sum_{i=1}^R(u_j\circ f_i)(v_j\circ g_i)^T.
\]
Each term has rank at most one, giving $\rank M\le R_B R$. If $F$ is a
polynomial, collect the coefficients of its monomials in $W$. These
coefficients are functions of $Z$. There are at most $\binom{m+\delta}{m}$
such monomials.
\end{proof}

\begin{theorem}[Diagonal separation]\label{thm:diagonal-separation}
Over $\overline{\F}_2$, take a constructible family
$\Ical\subseteq\A^L\times(\A^m)^s$ given by polynomial conditions over
$\F_2$. Suppose its tuples have distinct labels and
\begin{equation}\label{eq:critical-dimension}
\dim\Ical\le L+m.
\end{equation}
Write $Y$ for the image closure on the coefficients and the first two
labels. We can compute a polynomial $F(A,Z,W)$ over $\F_2$ with
$\deg F\le\cdeg Y$ that vanishes on $Y$ and has
\begin{equation}\label{eq:nonzero-diagonal}
D_F(A,Z):=F(A,Z,Z)\not\equiv0.
\end{equation}
If $Y$ is empty use $1$. For an integer $\delta\ge\max(1,\deg F)$ and
$q=2^h>2\delta$, the diagonal condition allows coefficient specialization
and decoding of $q^{m-1}(q-\delta)$ labels.

Consider a forbidden occurrence specified by required edges and fixed
data on its ordered tuple. More generally, suppose that whenever such
an occurrence exists after deleting vertices or edges, it exists on
the same tuple in the old graph. Require every old occurrence to be
represented in $\Ical$, with a required edge $z_1z_2$ and $z_1\prec z_2$.
Keeping the old edges where the specialized $F$ is nonzero excludes
all occurrences. For any certificate matrix $B$ of
Lemma~\ref{lem:filter-rank} available for each assignment, the new graph
has $\alpha\le\binom{m+\delta}{m}\rank B$. With fixed model parameters,
decoding and the additional polynomial evaluation at an edge have
deterministic bit cost polynomial in $h$ after preprocessing.
\end{theorem}

\begin{proof}
Project onto $(A,z_1,z_2)$ and call the image $S$. Its closure $Y$ has
dimension at most $L+m$. In the target consider the diagonal
$D=\{(A,Z,W):Z=W\}$. This is an irreducible affine space of dimension
$L+m$, and no pair in $S$ lies on it. If $D\subseteq Y$, a component
containing $D$ would have to equal it by dimension. The complement in
$D$ of the other components would then contain a nonempty open subset
of $Y$. The dense set $S$ would meet this open subset, contrary to
distinctness.

Choose $x\in D\setminus Y$. Lemma~\ref{lem:degree} gives an equation of
$Y$ nonzero at $x$ of degree at most $\cdeg Y$. The family and its
projection are defined over $\F_2$, so the descent part of that lemma
gives such an equation $F$ with binary coefficients. Since $F(x)\ne0$,
the diagonal restriction $D_F(A,Z)$ is a nonzero polynomial. If $Y$ is
empty we can take $F=1$.

Enumerate binary polynomials by degree, testing the diagonal restriction
symbolically. To test vanishing on $S$, adjoin $vF(A,z_1,z_2)-1=0$ to
the witness equations of $\Ical$ and check for the unit ideal as in
Section~\ref{sec:degree-model}. Vanishing on $S$ is equivalent to
vanishing on $Y$. The degree bound guarantees termination, even when
the search starts without a numerical bound for $\cdeg Y$.

A nonzero coefficient of $D_F$ as a polynomial in $Z$ remains nonzero
for a choice $A=a$ supplied by Lemma~\ref{lem:adaptive}. Since
$q>2\delta$, Lemma~\ref{lem:decoder} indexes $q^{m-1}(q-\delta)$
labels at which $F(a,z,z)\ne0$. On these labels keep an old edge
$z\prec w$ only when $F(a,z,w)\ne0$. A surviving occurrence would be an
old one and its required ordered pair would satisfy
$F(a,z_1,z_2)=0$, a contradiction. Lemma~\ref{lem:filter-rank}
gives the claimed independence bound. Once $F$ is fixed, the decoder,
specialization and the evaluation at a pair have polynomial bit cost in $h$.
\end{proof}

\subsection{The cost of computing the equations}\label{sec:search-cost}
We give a coarse bit bound for the flag families. For $3\le d\le s$,
write $\kappa_{s,d}/(s-1)=m/b$ in lowest terms, as in
Section~\ref{sec:scalars}. Then $b\le s-1$, $m\le b(d-1)$ and the number
$L$ of coefficient slots is $O(s^6)$ in both interpolation families.
Proposition~\ref{prop:substituted-degree} gives
$\log_2\mathcal D=O(s^3\log s)$. A filter of degree at most $\mathcal D$
in $L+2m$ variables has at most
\[
S=\binom{L+2m+\mathcal D}{L+2m}
\]
possible monomials. Thus there are at most $2^S$ binary candidates.

How much does one vanishing test cost? Include the witnesses of
\eqref{eq:inverse-witness}, the split-factor distinctness witnesses for
the two-stage family, and the variable $v$ for $vF-1$. The system has
$n=O(s^6)$ variables and polynomially many equations of degree at most
$u=\max(4,2e,\mathcal D+1)$. Omit zero equations, handle constants
directly and square each quadratic. This preserves the common zero set,
keeps the degree bound $u$ and removes degree two. By Koll\'ar's effective
Nullstellensatz \cite[Theorem~1.5]{Kollar}, an inconsistent system then
has a certificate $\sum_i g_i f_i=1$ with
$\deg(g_i f_i)\le u^{n+1}$. Equating coefficients decides its existence
by Gaussian elimination over $\F_2$. There are at most
$(1+u^{n+1})^n=2^{s^{O(1)}}$ monomials through this degree. Hence one
test costs $2^{s^{O(1)}}$ bit operations, and the entire search costs at
most $2^{2^{s^{O(1)}}}$. This bounds an implementation using linear
algebra for the unit-ideal tests. In Appendix~\ref{sec:fibers} the
$s m(s-1)$ projected equations have degree at most $\mathcal D$ and
at most $L+m+1\le L+2m$ variables each. The same estimate includes
their searches, coefficient extractions and product.

\section{Attaining the rational endpoint}\label{sec:scalars}
The incidence bound permits $m\le\kappa_{s,d}/(s-1)$ without scalar
restriction. The right-hand side can be fractional, but $m$ counts coordinates.
We can achieve the fractional ratio by taking the flag coordinates in an algebra
of dimension $b$ over the label field. This multiplies the codimension by $b$
and measures the label dimension per flag-field degree by $m/b$.
Fix $b\ge1$ and an irreducible $g_b(U)\in\F_2[U]$ of degree $b$. With
$E_0=\F_2[U]/(g_b)$ and $q=2^h$, set
\begin{equation}\label{eq:fixed-algebra}
E_q=E_0\otimes_{\F_2}\F_q=\F_q[U]/(g_b(U)).
\end{equation}
In the basis $1,U,\ldots,U^{b-1}$, multiplying two elements uses bilinear
forms with binary coefficients fixed by $g_b$.

\begin{lemma}\label{lem:scalars}
After scalar extension to $\overline{\F}_2$, the flag algebra is a product
of $b$ copies of $\overline{\F}_2$. In its base coordinates the chart has
dimension $b(2d-3)$ and the ordered clique conditions have codimension
$b\kappa_{s,d}$. Over $\F_{2^h}$ with $\gcd(b,h)=1$, the algebra $E_q$
in \eqref{eq:fixed-algebra} is a field of order $2^{bh}$.
\end{lemma}

\begin{proof}
Over the algebraic closure, $g_b$ has $b$ distinct roots because finite
fields are perfect. The Chinese remainder theorem identifies evaluation at
the roots with the product decomposition. A polynomial equation in the
algebra becomes the same equation in each factor. We therefore get $b$
independent copies of the chart and of its clique configurations, with
dimensions and codimensions adding.

Over $\F_{2^h}$, an irreducible factor has degree equal to the orbit
length of one of its roots under $x\mapsto x^{2^h}$. The binary Frobenius
cycles through the $b$ roots of $g_b$. Taking its $h$-th power gives orbits
of length $b/\gcd(b,h)$. If $\gcd(b,h)=1$, the orbit contains all the roots,
so $g_b$ remains irreducible and the quotient has $2^{bh}$ elements.
\end{proof}

A point has $b(d-1)$ affine coordinates over $\F_q$. Choose $m$ of
these positions and put the entries of the label there. Reading those
positions recovers the original vertex. Give the other $c=b(2d-3)-m$ chart coordinates independent
interpolating polynomials in the label, with $L=cM(m,s)$ coefficient
slots in all. Using the codimension from Lemma~\ref{lem:scalars} in
Lemma~\ref{lem:universal-dimension} gives
\begin{equation}\label{eq:scalar-incidence}
\dim\Ical\le L+ms-b\kappa_{s,d}.
\end{equation}
Here Lemma~\ref{lem:interpolation} is applied to distinct labels over
the base. The codimension calculation permits coincident points in any
of the split factors.

\begin{theorem}[Critical flag construction]\label{thm:critical}
Fix $s,d\ge3$ and $b,m\ge1$ with
\begin{equation}\label{eq:endpoint-condition}
m\le b(d-1),\qquad m(s-1)\le b\kappa_{s,d}.
\end{equation}
Put $\mathcal D=(2s)^{b\binom s2}$. Deterministic preprocessing gives an
integer $1\le\delta\le\mathcal D$. For every $h$ with $\gcd(b,h)=1$ and
$q=2^h>2\delta$, there is an indexed graph satisfying
\begin{equation}\label{eq:critical-graph}
|V|=q^{m-1}(q-\delta),\qquad\omega<s,\qquad
\alpha\le C d^{bh},\qquad C=\binom{m+\delta}{m}.
\end{equation}
In particular $q>2\mathcal D$ suffices and
$C\le\binom{m+\mathcal D}{m}$. With the other parameters fixed,
initialization, decoding and adjacency have deterministic bit cost
polynomial in $h$ after preprocessing.
\end{theorem}

\begin{proof}
The hypotheses and \eqref{eq:scalar-incidence} give $\dim\Ical\le L+m$.
Let $X$ be the closure of this incidence, and $Y$ the closure of its
projection to the coefficients and the first pair of labels. By
Proposition~\ref{prop:substituted-degree} with $e=s$ and the linear
projection bound,
\[
\cdeg Y\le\cdeg X\le (2s)^{b\binom s2}=\mathcal D.
\]
Use the scalar clique equations and the distinctness witnesses over
$\F_2$ in Theorem~\ref{thm:diagonal-separation}. The filter it computes
has $\delta=\max(1,\deg F)\le\mathcal D$. For $q>2\delta$, specialize $A$
so that $F(A,Z,Z)$ remains nonzero and index
$q^{m-1}(q-\delta)$ labels by Lemma~\ref{lem:decoder}.

Put the labels in lexicographic order by their binary base coordinates.
For $z\prec w$ keep the edge when $H(w)P(z)=0$ in $E_q$ and
$F(A,z,w)\ne0$ in $\F_q$, using the specialized coefficients.
The increasing tuple of any surviving clique would be an original
clique and its first edge would fail the filter.

By $\gcd(b,h)=1$, the algebra $E_q$ is the field $\F_{2^{bh}}$.
It contains the filter values. Applying Lemmas~\ref{lem:rank}
and~\ref{lem:filter-rank} in this common field gives the multiplier
in \eqref{eq:critical-graph}. The multiplication table of $E_0$ is
fixed, so the field operations, specialization and decoding needed for
a local query have polynomial bit cost in $h$.
\end{proof}

\begin{lemma}\label{lem:extraction}
Let $h$ run through an arithmetic progression with step $g\ge1$. Suppose that every sufficiently large such $h$ gives a $K_s$-free graph $G_h$ satisfying
\[
|V_h|\ge c_0 2^{mh},\qquad\alpha(G_h)\le C\theta^{bh},\qquad\theta>1,
\]
where $c_0,C,m,b$ are positive constants independent of $h$. Then
\begin{equation}\label{eq:bounded-gap-extraction}
R(s,t)\ge\Omega\bigl(t^{m/(b\log_2\theta)}\bigr).
\end{equation}
\end{lemma}

\begin{proof}
Take the largest degree $h$ in the progression for which $C\theta^{bh}\le t-1$.
For large $t$ such an $h$ lies in the range of the construction.
The next degree in the progression is $h+g$ and fails the independence
inequality, so
\[
h>\frac{\log_2((t-1)/C)}{b\log_2\theta}-g.
\]
It follows that $|V_h|\ge c_0 2^{-mg}((t-1)/C)^{m/(b\log_2\theta)}$.
Since $\omega(G_h)<s$ and $\alpha(G_h)<t$, we have $R(s,t)>|V_h|$. The factor $2^{-mg}$ depends only on the fixed parameters.
\end{proof}

Take $s=16,d=13$ first. The codimension is $114$. With $b=2$
and $m=15$, the incidence dimension in \eqref{eq:scalar-incidence}
is at most $L+15\cdot16-2\cdot114=L+12$. The coefficients and
the first label lie in a space of dimension $L+15$. Hence there
is a nonzero equation for their image. This is the strict marked
construction after restriction of scalars.

In general, let $1\le m\le b(d-1)$ and $m(s-1)<b\kappa_{s,d}$.
Projecting to $(A,z_1)$ gives an image closure of dimension at
most $L+ms-b\kappa_{s,d}<L+m$. By Lemma~\ref{lem:degree} it has
a nonzero equation $Q(A,Z)$ over $\F_2$. Use Lemma~\ref{lem:adaptive}
on a nonzero coefficient in $Z$ to arrange $Q(a,Z)\not\equiv0$.
For large $q=2^h$, Lemma~\ref{lem:decoder} provides an indexed
subset of size at least $q^m/2$ where $Q(a,z)\ne0$. The first
vertex of any ordered clique would violate this condition. The
subgraph is induced, so $\alpha\le d^{bh}$ for $\gcd(b,h)=1$.
Use the progression $h\equiv1\pmod b$ in Lemma~\ref{lem:extraction}
to obtain
\begin{equation}\label{eq:marked-scalars}
R(s,t)\ge\Omega_{s,d,b,m}\bigl(t^{m/(b\log_2d)}\bigr).
\end{equation}
In our example $225<228$, $h$ is odd and the exponent is
$15/(2\log_2 13)=2.0267861582\ldots$.

At equality choose $m/b=\kappa_{s,d}/(s-1)=\mu_{s,d}$ in lowest terms.
Since $\mu_{s,d}\le d-1$, the point coordinates contain $m$ base
coordinates for the label. The progression $h\equiv1\pmod b$ has step $b$. Applying
Theorem~\ref{thm:critical} and Lemma~\ref{lem:extraction} gives
\begin{equation}\label{eq:critical-exponent}
R(s,t)\ge\Omega_{s,d}\bigl(t^{\mu_{s,d}/\log_2d}\bigr).
\end{equation}
For $d\ge s$, the codimension is $\binom s2$ and the rank base $d$
increases. It suffices to take $3\le d\le s$, so
\begin{equation}\label{eq:Ecrit}
E_{\rm crit}(s)=\max_{3\le d\le s}
\frac{(d-1)(2s-d)}{2(s-1)\log_2d}.
\end{equation}
At $d=s$ the quotient equals $s/(2\log_2s)$.

\section{Reducing the rank with Frobenius relations}\label{sec:words}
In Lemma~\ref{lem:rank}, each binary digit contributes one of the $d$ coordinates of $P$. The bound $d^H$ counts all choices separately. A relation between the coordinates can make some of their products linearly dependent. Write
\[
P=(1,X_1,\ldots,X_a,Y_1,\ldots,Y_r),\qquad d=1+a+r,
\]
where $a\ge1$, and put $X_0=1$. Suppose every $Y_i$ is a linear combination
over the flag field of the products
\begin{equation}\label{eq:short-relation}
\prod_{j=0}^{k-1}X_{\sigma(j)}^{2^j},\qquad \sigma:\{0,\ldots,k-1\}\longrightarrow\{0,\ldots,a\},\qquad k\ge2.
\end{equation}
The same estimate will hold for all choices of coefficients in these relations.

\begin{lemma}[Word reduction]\label{lem:words}
If $r>0$, let $\rho=\rho(d,r,k)$ be the largest positive root of
\begin{equation}\label{eq:rho}
x^k-dx^{k-1}+r=0.
\end{equation}
This root lies in $(d-1,d)$. Over $\F_{2^H}$ the relations
\eqref{eq:short-relation} give $\dim\Wcal_H(P)\le d^k\rho^H$ and hence
the same bound on $\alpha$. For $r=0$ write $\rho=d$ and use
$\dim\Wcal_H(P)\le d^H$.
\end{lemma}

\begin{proof}
Write a feature as a word of length $H$ on the $d$ coordinate letters. The
letter in position $j$, starting from $0$, contributes its $2^j$-th power.
Suppose a word contains $Y_i0^{k-1}$, where $0$ denotes the constant coordinate.
Its contribution starting at position $j$ is $Y_i^{2^j}$. Raising
\eqref{eq:short-relation} to the $2^j$-th power replaces this part by a
linear combination of products in the $k$ positions $j,\ldots,j+k-1$,
all in the $X$ coordinates. Each replacement removes one $Y$ and introduces
none. Hence the features are spanned by the words that avoid all these
blocks. Let $u_n$ count such words of length $n$.

Appending a letter to each word counted by $u_{n-1}$ gives $du_{n-1}$ candidates in total.
The only new failure is a terminal $Y_i0^{k-1}$. These words consist of a
word counted by $u_{n-k}$ followed by one of the $r$ forbidden blocks.
There is no crossing occurrence before the terminal block, since its
initial $Y_i$ cannot be one of the zeros of an earlier block. Thus
\[
u_n=d^n\quad(0\le n<k),\qquad u_n=du_{n-1}-ru_{n-k}\quad(n\ge k),
\]
and $\sum_{n\ge0}u_nz^n=(1-dz+rz^k)^{-1}$. To get a bound with a fixed
prefactor, keep track of the last letters with a finite transition matrix.
There is a safe state and one state for a terminal $Y$ followed by $j$
zeros for each $0\le j\le k-2$. An $X_i$ with $i\ge1$ returns to the safe
state and a $Y_i$ begins a new terminal block. A zero leaves the safe state
unchanged or advances the terminal block, with the last transition
forbidden. The matrix is primitive. Every state can reach the safe state
via an $X_i$, every terminal state can be reached from the safe state, and
the safe state has a loop. Consequently its Perron root is the growth rate
of $u_n$ and is a root of \eqref{eq:rho}.

The polynomial in \eqref{eq:rho} is negative at $d-1$ since
$r\le d-2<(d-1)^{k-1}$, and it is positive at $d$. On the positive axis
its derivative has sign given by $kx-d(k-1)$, so it first decreases and then
increases. It therefore has a unique root above $d-1$. The row sums of the
transition matrix are $d-1$ and $d$, both occurring, so its Perron root lies
strictly between those values by irreducibility. This identifies it with
$\rho$.

Let $T$ be the transition matrix and $v>0$ its right Perron vector.
From a state $i$ we can reach any state $j$ in at most $k$ steps by way
of the safe state. If such a path has length $u$, then $(T^u)_{ij}\ge1$,
so $\rho^uv_i=(T^uv)_i\ge v_j$. Thus $\max_jv_j/\min_iv_i\le\rho^k\le d^k$. The number $u_H$ is the entry at
the safe state of the $H$-th matrix power applied to the vector of ones.
Comparing this vector with $v/\min_i v_i$ gives $u_H\le d^k\rho^H$.
Now apply Lemma~\ref{lem:rank}.
\end{proof}

To obtain \eqref{eq:short-relation} while interpolating at $s$ inputs,
use the coefficients of $\ell_T(X)^j$ for $j<s$. As $s\le2^k$, the binary
expansion of $j$ gives a product of the factors
$\ell_T(X)^{2^u}$ for the occupied binary digits $u$ of $j$. Each such
factor uses one point coordinate to the $2^u$-th power. Insert $X_0^{2^u}=1$
at an unoccupied digit. Extracting a coefficient of $T$ leaves a
linear combination of the products in \eqref{eq:short-relation}. We
therefore interpolate $Y$ in $X$. Interpolation of both in the original
labels alone would give no such relation between the point coordinates.

\subsection{Interpolating first the labels and then the flags}\label{sec:two-stage}
We need enough coordinates in $X$ to retain the label, and we want the other
point coordinates to be polynomials in $X$. Fix $3\le d\le s$, write
$\mu_{s,d}=m/b$ in lowest terms, and choose
\begin{equation}\label{eq:two-stage-parameters}
a=\lceil\mu_{s,d}\rceil,\qquad r=d-1-a,\qquad k=\lceil\log_2s\rceil.
\end{equation}
One coordinate in $E_q$ has $b$ base coordinates. Hence $a$ is the fewest
coordinates in $E_q$ that can carry the full label. Copy the base label
coordinates into $m$ of the $ba$ coordinates of $X\in E_q^a$ and interpolate
the other $ba-m$ by independent compressed polynomials in the labels.
Having chosen $X$, interpolate the $r$ coordinates of $Y$ and the $d-2$
slopes by independent compressed polynomials in $X$, with coefficients
in $E_q$. The number of base-field coefficient slots is
\begin{equation}\label{eq:two-stage-L}
L=(ba-m)M(m,s)+b(r+d-2)M(a,s).
\end{equation}
Each $Y_i$ has the form required by \eqref{eq:short-relation}. For the
word argument we use $X$ as the input, so the complexity of $X$ as a
polynomial in the labels does not enter the bound.

\begin{proposition}\label{prop:two-stage}
In the two-stage family, restrict the coefficient--label incidence to
ordered cliques at distinct base labels whose $X$ values are pairwise
distinct in each of the $b$ geometric factors. Then its dimension is at most
\[
L+ms-b\kappa_{s,d}\le L+m.
\]
When $E_q$ is a field, every ordered clique on distinct labels belongs to
this incidence. Consequently there are $K_s$-free graphs on
$q^{m-1}(q-\delta)$ vertices with
\[
\alpha\le C d^k\rho^{bh},
\]
for fixed $C,\delta$ and every sufficiently large $h$ with $\gcd(b,h)=1$.
For fixed $s,d$, preprocessing is deterministic and terminating, and
local access takes time polynomial in $h$.
\end{proposition}

\begin{proof}
Over the algebraic closure, work in the $b$ split factors of the flag
algebra. Take the open configuration set where the $X$-tuples are
pairwise distinct in each factor and the base labels are pairwise distinct.
Fix a full configuration in this open set. The labels and $X$ values have
both been fixed. The label evaluations give $s$ independent linear
conditions on each coefficient block for the $ba-m$ first-stage outputs.
Likewise the fixed $X$ evaluations give $s$ independent linear conditions
on each second-stage block in each factor, by Lemma~\ref{lem:interpolation}.
To count the second-stage constraints in base coordinates, apply the
splitting isomorphism to every coefficient block. On a column of $b$
coordinates it acts by the matrix $(\xi_i^{\,j})_{1\le i\le b,\,0\le j<b}$
at the roots of $g_b$, whose determinant is nonzero. After this change
of variables the second-stage evaluation matrix consists of $b$ independent
blocks, each of rank $s$ per output. The fiber of coefficient choices has dimension
\[
L-s\bigl((ba-m)+b(r+d-2)\bigr)=L-s\bigl(b(2d-3)-m\bigr).
\]
On nonzero-minor charts for the evaluation matrices these coefficient
fibers are affine spaces with free coordinates, as in
Lemma~\ref{lem:universal-dimension}. The configuration set has dimension
at most $s b(2d-3)-b\kappa_{s,d}$. Adding the fiber dimension proves the bound.

For preprocessing, express distinctness of the $X$-tuples in all factors
by one equation in the algebra for each pair $i<j$. Introduce
$U_{ij,1},\ldots,U_{ij,a}$ and require
\[
\sum_{\ell=1}^a U_{ij,\ell}(X_{i,\ell}-X_{j,\ell})=1\quad\text{in }E_q.
\]
In a split factor the equation is solvable precisely when some difference
coordinate is nonzero there. Use \eqref{eq:inverse-witness} for the base
labels. These equations describe the open set by projection. The
dimension calculation fixes the configurations themselves and counts
their coefficient fibers.

In the actual field $E_q$, the copied base coordinates ensure that distinct
labels give distinct $X$-tuples. Each embedding of $E_q$ into the
algebraic closure is injective, so a nonzero difference between two tuples
remains nonzero in every factor. Thus the restriction retains all actual
ordered cliques. Apply Theorem~\ref{thm:diagonal-separation} to delete an
edge from each one. Lemmas~\ref{lem:filter-rank} and \ref{lem:words}
give the bound on independence. This proof applies to every sufficiently large $h$ with $\gcd(b,h)=1$.
\end{proof}

\begin{corollary}\label{cor:frobenius-exponent}
With $\rho$ defined by \eqref{eq:two-stage-parameters} and \eqref{eq:rho}, put
\begin{equation}\label{eq:EF}
E_{\rm F}(s)=\max_{3\le d\le s}\frac{\mu_{s,d}}{\log_2\rho}.
\end{equation}
With deterministic preprocessing for fixed $s$ and local access polynomial
in the extension degree, the two-stage family gives $R(s,t)\ge\Omega_s(t^{E_{\rm F}(s)})$. At $s=16$ choose $d=13$. Then $\mu=38/5$, so $a=8$, $r=4$ and $k=4$. Hence
\begin{equation}\label{eq:R16}
R(16,t)\ge\Omega\left(t^{\frac{38}{5\log_2\rho_{16}}}\right),\qquad
\rho_{16}^4-13\rho_{16}^3+4=0,\quad 12<\rho_{16}<13.
\end{equation}
The exponent is $2.0539221767\ldots$.
\end{corollary}

\begin{proof}
Apply Lemma~\ref{lem:extraction} to Proposition~\ref{prop:two-stage}, absorbing
$d^k$ into the constant. For $s=16,d=13$ the codimension is
$\kappa_{16,13}=114$, so we copy $m=38$ label coordinates into $ba=40$
base coordinates and use a flag field of degree $H=5h$ with $h\equiv1\pmod5$.
The vertex set has at least $2^{38h-1}$ elements for large $h$. The
factor $C d^k$ in the independence bound is independent of $h$, so
the exponent supplied by Lemma~\ref{lem:extraction} is
$38/(5\log_2\rho_{16})$.
\end{proof}

\begin{proposition}\label{prop:asymptotic}
As $s\to\infty$,
\[
E_{\rm crit}(s)=\frac{s}{2\log_2s}\left(1+\frac{1}{4(\ln s)^2}+O((\ln s)^{-3})\right),
\qquad E_{\rm F}(s)=\left(\frac12+o(1)\right)\frac{s}{\log_2s}.
\]
An integer $d$ at which \eqref{eq:Ecrit} attains its maximum satisfies
$d=s-s/(2\ln s)+O(s/(\ln s)^2)$.
\end{proposition}

\begin{proof}
Allow $d$ to vary continuously in $[3,s]$. Since $\mu_{s,d}\le d-1$,
the quotient in \eqref{eq:Ecrit} is $o(s/\log s)$ uniformly for
$d\le s/(\ln s)^2$. For the remaining values of $d$,
$\ln s-2\ln\ln s\le\ln d\le\ln s$.
If $d/s\le1-\varepsilon$ for some fixed $\varepsilon>0$, the numerator
$(d-1)(2s-d)$ is at most $s^2(1-\varepsilon^2)$. Comparing with $d=s$
shows that any maximizing sequence has $d/s\to1$. Write $x=1-d/s$ and
$L_s=\ln s$. The numerator after division by $s(s-1)$ is
$1-x^2-(x+x^2)/(s-1)$. Up to a uniform relative $O(1/s)$ error near
$x=0$, the ratio to $s/(2\log_2s)$ is
\[
\frac{1-x^2}{1+\ln(1-x)/L_s}.
\]
The exact ratio has positive derivative at $x=0$ for large $s$, and its
second derivative is $-2+O(L_s^{-1})$ on a fixed small interval at $0$.
Thus the real maximum near $0$ is interior and unique. Replacing the
exact numerator by $1-x^2$ changes the derivative by $O(1/s)$ there.
The stationary equation for \eqref{eq:Ecrit} in this variable is therefore
$2x(L_s+\ln(1-x))=1+x+O(L_s/s)$, where $x=o(1)$. This gives
$x=1/(2L_s)+O(L_s^{-2})$. Substituting in the displayed ratio gives
$1+1/(4L_s^2)+O(L_s^{-3})$. Near the maximum the exact ratio
increases up to its critical point and decreases after it. The integer
maximum is attained at one of the two adjacent mesh points, with mesh
size $1/s$, which preserves both expansions.

Since $\rho\le d$, we have $E_{\rm F}\ge E_{\rm crit}$. For an upper
bound, use $\rho\ge d-1$ and $\mu_{s,d}\le\min(d-1,s/2)$. In the range
$d\le s/(\ln s)^2$ the quotient is $o(s/\log s)$, by monotonicity of
$x/\log_2 x$ for large $x$ and a constant bound for bounded $d$.
In the other range $\log_2(d-1)=(1+o(1))\log_2s$ uniformly and the
numerator is at most $s/2$.
\end{proof}

\subsection{Computing the rank of a fixed block}\label{sec:finite-blocks}
For a second bound on the columns of the incidence matrix, choose a
length $\ell$ and expand every word of that length as a polynomial.
Their coefficient matrix has finite size. We shall see how its rank
bounds words of every length $H$.

\begin{proposition}\label{prop:finite-block}
For a fixed point family, let $\beta_n=\dim\Wcal_n$. Then
$\beta_{u+v}\le\beta_u\beta_v$ and
$\eta=\lim_{n\to\infty}\beta_n^{1/n}=\inf_n\beta_n^{1/n}$ exists.
For every fixed $\ell\ge1$,
\[
\beta_H\le d^{\ell-1}\beta_\ell^{\lfloor H/\ell\rfloor}.
\]
For a polynomial family let $\beta_\ell^{\rm gen}$ denote the rank of
the length-$\ell$ feature polynomials when their coefficients are
indeterminates, with rank taken over the rational function field in
these parameters. This bounds the rank after any specialization and
evaluation. If the same flag family has a critical construction with
$m/b=\mu$ and a fixed rank multiplier as in
Theorem~\ref{thm:critical}, then $\beta_\ell^{\rm gen}>1$ gives exponent
$\mu\ell/\log_2\beta_\ell^{\rm gen}$.
\end{proposition}

\begin{proof}
The first $u$ positions of a feature give a function in $\Wcal_u$.
The last $v$ positions give the $2^u$-th power of a function in
$\Wcal_v$. If $g_1,\ldots,g_{\beta_v}$ span $\Wcal_v$, their $2^u$-th powers
span all such last factors, since the map $x\mapsto x^{2^u}$ preserves
sums and products. Multiply these powers by a basis of $\Wcal_u$ to
obtain at most $\beta_u\beta_v$ spanning functions.
Subadditivity of $\log\beta_n$ gives the limit. Splitting $H$ into
blocks of length $\ell$ and a remainder of length at most $\ell-1$
gives the explicit prefactor in the finite inequality.

For the generic rank, expand the feature polynomials in the input
monomials and use their coefficient vectors as columns. The entries
of this finite matrix are polynomials in the parameters. All minors of size $\beta_\ell^{\rm gen}+1$ vanish identically in the
parameters. They therefore still vanish after the coefficients have
been chosen, and evaluating the new columns at the vertex labels can
only decrease their rank. Thus the same rank bound holds for the
coefficient choice used to avoid cliques. With $\ell$ fixed, the finite
inequality bounds $\alpha$ by a fixed multiple of
$(\beta_\ell^{\rm gen})^{H/\ell}$. Lemma~\ref{lem:extraction} gives
the exponent.
\end{proof}

For an example, take $P=(1,X,u+vX+wX^2)$ with $u,v,w$ independent
parameters over $\F_2$. Over $\F_2(u,v,w)$ the three entries span
$\langle1,X,X^2\rangle$. Their squares span $\langle1,X^2,X^4\rangle$.
Hence the two-digit products span the monomials $X^j$ with
$j\in\{0,1,2\}+2\{0,1,2\}=\{0,\ldots,6\}$. These monomials are
independent in the polynomial ring, so $\beta_2^{\rm gen}=7$ for
nine products. This computation gives a rank estimate for the point
family. To extract a Ramsey exponent one would also need a clique
incidence bound for flags with these points.

\section{Other patterns and other prime fields}\label{sec:general}
In a clique the first two positions always give an edge. For another
pattern, specify two positions $u<w$ that have to form an edge in
each occurrence, and give algebraic conditions on the tuple that hold
also for occurrences after edges have been deleted. These conditions
can overcount the tuples. We need an upper bound on the dimension of
the family they define in $(\A^\nu)^v$ for each order of a pattern on $v$ vertices.

\begin{proof}[Proof of Theorem~\ref{thm:configuration}]
Work over $\overline{\F}_2$ to count dimensions. Index each possible
order separately and let $i$ refer to one ordered family. Splitting
the degree-$b$ algebra gives $b$ copies of $\A^\nu$. The product
of the $b$ copies of $\Ccal_i$ has dimension at most
$b(\nu v_i-\kappa_i)$. A rational occurrence in the flag field belongs
to each copy under the corresponding embedding, by the hypothesis on
the family. Restrict to distinct base labels. Use the features of
Section~\ref{sec:interpolation} through degree $\max_i v_i-1$ and
independent coefficients for each of the $b\nu-m$ outputs.
For a fixed configuration, evaluation at its $v_i$ labels imposes
$v_i(b\nu-m)$ independent linear conditions. The dimension of the
coefficient--label incidence is therefore at most
\[
 b(\nu v_i-\kappa_i)+L-v_i(b\nu-m)
 =L+mv_i-b\kappa_i\le L+m .
\]
Its model over $\F_2$ comes from the fixed prime-field conditions
and restriction of scalars. Project to $(A,z_u,z_w)$ with $u<w$
the positions of the distinguished edge in this order.
Theorem~\ref{thm:diagonal-separation} gives an equation $F_i$ of
the pair image with $F_i(A,Z,Z)\not\equiv0$.

Take the product $F=\prod_i F_i$. Its diagonal restriction is a
product of nonzero polynomials. Specialize a nonzero coefficient of
this restriction by Lemma~\ref{lem:adaptive}. If $\delta\ge1$ bounds
$\deg F$ and $q=2^h>2\delta$, Lemma~\ref{lem:decoder} indexes
$q^{m-1}(q-\delta)\ge q^m/2$ labels with $F(a,z,z)\ne0$.
Keep an old edge $z\prec w$ precisely when $F(a,z,w)\ne0$.
A surviving occurrence would, by hypothesis, give a tuple in one
of the original ordered families with its distinguished edge still
present. Its factor $F_i$ vanishes there, a contradiction.

For $\gcd(b,h)=1$ the flag algebra is a field. In the common field
of the certificate and the filter values, Lemma~\ref{lem:filter-rank}
bounds the independence number by $\binom{m+\delta}{m}R(bh)$.
The decoder and evaluation of the fixed filter use polynomial bit
time in $h$ for fixed parameters. An adjacency query also calls the
original graph's edge test.
\end{proof}

\subsection{A fixed prime characteristic}
Fix a prime $p$. Interpolation, scalar restriction, diagonal separation
and decoding can be done over $\F_p$ with the same dimension and degree
arguments. We still take the varying extension degree coprime to $b$.
The change in the Ramsey exponent comes from the rank estimate. At
each base-$p$ digit we now choose a point monomial of degree $p-1$.

\begin{proposition}\label{prop:prime}
For a flag field of order $p^H$, put $B_{d,p}=\binom{d+p-2}{p-1}$.
For every assignment of chart flags the ordered graph has a matrix
of rank at most $B_{d,p}^H$, with diagonal $1$ and with a zero at
each increasing nonedge. For fixed $s,d,p$ the critical construction gives
\[
R(s,t)\ge\Omega_{s,d,p}\left(t^{\kappa_{s,d}/((s-1)\log_p B_{d,p})}\right).
\]
\end{proposition}

\begin{proof}
Use the base-$p$ digits of $p^H-1$ to write
\[
K^{p^H-1}=\prod_{j=0}^{H-1}\bigl(K^{p-1}\bigr)^{p^j}.
\]
Every term of $K^{p-1}$ has a point part indexed by nonnegative
integers $i_1,\ldots,i_d$ with sum $p-1$. There are $B_{d,p}$ such
tuples. The coefficient $(p-1)!/\prod_j i_j!$ is nonzero modulo $p$,
and the remaining factor is a function of the hyperplane alone.
The constant point coordinate supplies the constant feature, as in
Lemma~\ref{lem:rank}. The product therefore puts every column of
$1-K^{p^H-1}$ in a span of at most $B_{d,p}^H$ functions. On an independent
set the upper entries are zero and the diagonal entries are $1$.

For descent use $p$-th powers on coefficients. In the decoder split
the rational roots by the tests
$\Tr_{\F_{p^h}/\F_p}(\beta_jX)-c$ for $c\in\F_p$. Nondegeneracy of the
trace pairing guarantees separation of any two roots. Since $p$ is fixed,
these tests retain polynomial bit cost in $h$. Use
$m/b=\kappa_{s,d}/(s-1)$ and $h\equiv1\pmod b$ in the scalar construction.
The vertex count grows like $p^{mh}$ and the independence bound like
$B_{d,p}^{bh}$ up to constants. The bounded-gap argument in
Lemma~\ref{lem:extraction}, with base $p$, gives the displayed exponent.
\end{proof}

\section{From an index to an edge}\label{sec:recipe}
In the direct graph the $m$ field entries of a vertex form its index.
For the endpoint graph choose $s,d$, the ratio $m/b$ and the single-
or two-stage family. Before choosing $h$, compute either the pair
filter of Theorem~\ref{thm:diagonal-separation} or the equations
for the finite fibers in Appendix~\ref{sec:fibers}.

Given a sufficiently large $h$ coprime to $b$, form $\F_q$ with
$q=2^h$ and the algebra $E_q$ from the fixed multiplication table.
Choose coefficients keeping the relevant diagonal or marked polynomial
nonzero. A vertex index in \eqref{eq:product-labels} then recovers its
label by the decoder. For two indices, order the labels and evaluate
the ordinary incidence. In the pair-filter construction an incident
pair requires one more evaluation, $F(a,z,w)$. In the finite-fiber
construction we compute both conflict colors and keep the edge when
they agree. A largest color class alone gives the Ramsey exponent by
its size, but the coloring does not give a procedure for indexing that
class. The graph using all colors retains the product index set.
For fixed parameters, initialization and local queries in either
construction have deterministic bit cost polynomial in $h$.

The dimension of the full clique configuration space gives the ratio
$\kappa_{s,d}/(s-1)$ by the first-moment calculation and by the
polynomial filter. To raise the exponent within this approach one
can restrict the point families while retaining interpolation, or
lower the rank of the certificate at the same ratio. Beyond the
dimension condition of Theorem~\ref{thm:diagonal-separation}, a filter
would still suffice if the closure of the pair image could be shown
to miss a point of the diagonal.

For the family in Proposition~\ref{prop:two-stage}, finite feature
matrices offer another way to approach the rank. A length $\ell$ with
$(\beta_\ell^{\rm gen})^{1/\ell}<\rho$ would improve its exponent.
For a different family one would also have to establish the incidence
bound that yields the $K_s$-free graph.

\paragraph{Acknowledgements and declaration of assistance.}
The author thanks Sam Mattheus for bringing this problem to his attention.
Most of the ideas in the paper originated in exchanges with GPT-6 Pro.
GPT-6.1 sol assisted with writing and revision in Codex. Codex also
assisted with drafting and reviewing the present revision. The author
takes full responsibility for the contents of the paper.

\appendix
\section{Indexing a polynomial complement}\label{sec:decoder}
To recover a label from its index, we first choose a line and then a
point on that line. We need a direction for which each parallel line
loses at most $\Delta$ points to the zero set. The part of highest
degree of the polynomial supplies such a direction.

\begin{lemma}\label{lem:decoder}
Let $m\ge1$ and $\Delta\ge0$ be integers. Suppose
$F\in\F_q[Z_1,\ldots,Z_m]$ is nonzero of degree at most $\Delta$,
where $q=2^h>\Delta$. There is an injective map from
$\F_q^{m-1}\times\{0,\ldots,q-\Delta-1\}$ to $\{z:F(z)\ne0\}$.
For fixed $m,\Delta$, its value at an index and the decision whether
a label belongs to its image can both be computed deterministically
in time polynomial in $h$.
\end{lemma}

\begin{proof}
Let $\delta=\deg F$. If $\delta=0$, we can use the first $q-\Delta$
encoded field elements on each last-coordinate fiber. Suppose
$\delta>0$ and write $F_\delta$ for the highest homogeneous part.
The grid of the first $\delta+1$ field elements in each coordinate
contains a vector $u$ with $F_\delta(u)\ne0$. Otherwise applying the
univariate root bound successively to the coordinates would give
$F_\delta=0$. The grid size depends only on the fixed parameters.
Since $u\ne0$, extend it to a basis by testing the standard basis
vectors in order. Let $S$ be the invertible matrix of this basis
with last column $u$.

Write $Z=S(y,x)^T$ with $y\in\F_q^{m-1}$. In
$\widehat F(y,x)=F(S(y,x)^T)$ the coefficient of $x^\delta$ is
$F_\delta(u)$, independent of $y$. Thus for every $y$ the univariate
polynomial $f(X)=\widehat F(y,X)$ has degree $\delta$ and at most
$\Delta$ roots in the field.

To list those roots, compute
\[
R(X)=\gcd(f(X),X^q-X).
\]
Each field element is a root of $X^q-X$, and the derivative is $-1$.
Thus $R$ has a single linear factor for every distinct field root of
$f$, even when a root of $f$ is repeated. For the gcd we compute
$X^{2^h}\bmod f$ with $h$ squarings. If $R=1$, the line contains no
zero of $F$.

Choose a binary basis $\beta_1,\ldots,\beta_h$ of $\F_q$. For each $j$,
split every current factor by taking gcds with
\begin{equation}\label{eq:trace-split}
T_j(X)=\sum_{\ell=0}^{h-1}(\beta_jX)^{2^\ell}
\quad\text{and}\quad T_j(X)-1,
\end{equation}
computed modulo that factor. At a field element $x$, $T_j(x)$ is
$\Tr(\beta_jx)\in\{0,1\}$. To see that these tests separate field elements, observe that
$T(X)=X+X^2+\cdots+X^{2^{h-1}}$ has degree $2^{h-1}<q$ and is
nonzero. Some $w\in\F_q$ therefore satisfies $T(w)=1$. For every
nonzero $v$, the linear map $u\mapsto\Tr(uv)$ takes the value $1$
at $u=w/v$, hence is nonzero on at least one of the basis elements. Apply this to the difference of two
roots. Their trace values differ for at least one of the tests, so
after $h$ tests every remaining factor is linear. The degrees are
bounded by $\Delta$, and the calculation has polynomial bit cost in
$h$ for fixed $\Delta$.

Identify field elements with the integers $0,\ldots,q-1$ using
their binary polynomial-basis encodings. We take indices in
\begin{equation}\label{eq:product-labels}
\F_q^{m-1}\times\{0,1,\ldots,q-\Delta-1\}.
\end{equation}
For $(y,j)$, sort the root indices for the fiber at $y$. Start at
$x=j$. For each forbidden index in increasing order, increment $x$
if that index is at most $x$. Each increment skips one root that
would otherwise have been counted among the first available entries.
The final $x$ is the index of the $j$th nonroot, counted from zero, and
$x\le j+\Delta\le q-1$. Return $S(y,x)^T$. Invertibility of $S$ and
the separate indexing on each fiber make this map injective.

To decide whether a label belongs to the chosen set, apply $S^{-1}$. If the last coordinate is a
nonroot, subtract the number of smaller root indices from its
integer index. The result is its nonroot index, which has to be less
than $q-\Delta$ for the label to belong to the chosen set. The
case $m=1$ uses an empty tuple for $y$ and just one fiber.
\end{proof}

The same direction works for every index. After that direction has
been fixed, recovering the $j$th retained element on a line requires
only the distinct roots on that line. There are at most $\Delta$.

\section{A sufficient field size}\label{sec:fibers}
The polynomial edge test of Theorem~\ref{thm:diagonal-separation} reaches equality in the dimension bound. Here we reach the same ratio by counting the tuples through each vertex. A bound on the size of these fibers gives a bound on the number of colors, and hence on the multiplier in the independence estimate.

\begin{lemma}[Finite marked fibers]\label{lem:finite-fibers}
Let $X\subseteq\A^L\times(\A^m)^s$ be a reduced prime-field algebraic set with
\[
\dim X\le L+m,\qquad\cdeg X\le D.
\]
There is a nonzero $Q(A,Z)\in\F_2[A,Z]$ of degree at most
\begin{equation}\label{eq:fiber-Delta}
\Delta=s m(s-1)D
\end{equation}
such that for each $(a,z)$ with $Q(a,z)\ne0$ we have
\[
 \#\{(a,z_1,\ldots,z_s)\in X:z_j=z\}\le D\qquad (1\le j\le s).
\]
Here the count is over the algebraic closure. From a fixed prime-field description of $X$, or of an incidence whose image is dense in $X$, one can compute $Q$ deterministically.
\end{lemma}

\begin{proof}
Fix a marked position $j$ and a scalar coordinate $Y$ of another
label. Project $X$ to $(A,Z_j,Y)$. Since this target has dimension
$L+m+1$, the image closure is proper. It has cumulative degree at most
$D$ by Lemma~\ref{lem:degree}. That lemma also gives a prime-field
equation of degree at most $D$ for the closure. Write it as a
polynomial in the one coordinate $Y$.
\[
G_{j,Y}(A,Z,Y)=\sum_{u=0}^{u_{j,Y}}g_{j,Y,u}(A,Z)Y^u
\]
There are $s$ choices of the marked position and $m(s-1)$ choices for the coordinate $Y$. From each equation take its leading coefficient as a polynomial in $Y$ and multiply them. The product $Q$ is nonzero. Its degree is at most the sum of the factor degrees, at most $sm(s-1)D$.
If $Q(a,z)\ne0$, every equation $G_{j,Y}(a,z,Y)$ is a nonzero
univariate polynomial.

Consider a point in a fiber with the marked label set to $z$.
Each coordinate of every other label is a root of its
specialized equation. A nonzero constant equation makes this
fiber empty. Otherwise the fiber is contained in a finite
Cartesian product of root sets. This finite fiber is a linear
section of $X$, so \eqref{eq:bezout} gives cumulative degree at
most $D$. Counting distinct geometric points gives the asserted
bound, irrespective of scheme-theoretic multiplicities.

To compute the equations, enumerate nonzero binary polynomials of degree
at most $D$ for each projection and test vanishing on the
incidence model. Vanishing on the projected incidence implies
vanishing on its closure. The degree and descent bounds guarantee
termination, and the remaining operations are coefficient extraction
and multiplication.
\end{proof}

Specialize $Q$ by Lemma~\ref{lem:adaptive}. With $q>2\Delta$,
Lemma~\ref{lem:decoder} gives an indexed set $S$ of
$q^{m-1}(q-\Delta)$ labels in its complement. On $S$, join two
vertices in an auxiliary graph if they occur together in an
ordered $s$-clique of the original graph on $S$. Call this the
conflict graph. A chosen vertex can occupy any of $s$ positions.
For each position there are at most $D$ tuples through it, each
providing at most $s-1$ neighbors. Thus its degree is at most
\begin{equation}\label{eq:Lambda}
\Lambda=s(s-1)D.
\end{equation}
To list the neighbors of $z$, place it successively at each $j$. For every remaining coordinate $Y$, list the roots in $\F_q$ of $G_{j,Y}(a,z,Y)$ as in Appendix~\ref{sec:decoder}. Test each tuple in the product of these lists for the clique equations, increasing order, distinct labels and inclusion of its vertices in $S$. Collect its vertices other than $z$ and remove duplicates. The number of tuples tested is bounded in terms of the fixed parameters, so the cost is polynomial in $h$.

\begin{lemma}[Local color reduction]\label{lem:color}
Suppose that a graph on labels from $\F_{2^h}^m$ has degree at most $\Lambda$ and that its neighbor list at a given vertex can be computed in time polynomial in $h$ for fixed $m,\Lambda$. There is a deterministic proper coloring using at most
\begin{equation}\label{eq:palette}
C=(16(\Lambda+1))^{16}
\end{equation}
colors. For fixed $m,\Lambda$ the color at a given vertex takes polynomial bit time in $h$ to compute.
\end{lemma}

\begin{proof}
Initially give each vertex its binary label as a color, with common
palette bound $N_0=2^{mh}$. Suppose we have a proper coloring with
bound $N\ge2$. Put $\ell=\lceil\log_2N\rceil$ and choose the smallest
power of two $p\ge2$ with $p>\Lambda\ell$. Then
$p\le2(\Lambda\ell+1)$. Fix a bijection from the $p$ digits to the
elements of $\F_p$ that sends zero to zero. Use it to turn each color's base-$p$ expansion
into the coefficients of a polynomial $f$ of degree at most $\ell$.
For a given vertex, subtract each neighbor's polynomial from its own. Each difference is nonzero and has at most $\ell$ roots,
so at most $\Lambda\ell$ field elements are roots of any of them.
Take the first $x$ outside their union and use the pair $(x,f(x))$
as the new color.

If adjacent vertices received the same pair, their old polynomials
would agree at the common $x$, which was excluded. The new coloring
is proper and has palette bound $p^2$. For $N\ge2$,
\[
p^2\le16(\Lambda+1)^2(\log_2N)^2.
\]
For $N\ge C$ this is at most $\sqrt N$. Indeed, write
$A=\Lambda+1$. At $N=C$ the left side is
$4096A^2(4+\log_2 A)^2\le 4096A^2(4A)^2\le(16A)^8=\sqrt C$.
The ratio $(\log_2N)^2/\sqrt N$ decreases for $N\ge C$.
While $N>C$, taking logarithms shows that one round at least halves $\log_2 N$. Starting at $\log_2 N_0=mh$, we reach $N\le C$ in $O(\log h)$ rounds for fixed $m,\Lambda$.

A round takes as input the preceding colors of the vertex and
its neighbors. Recursing for $O(\log h)$ rounds needs at most
$(\Lambda+1)^{O(\log h)}$ oracle calls. The fields used in these
rounds have size polynomial in $h$, so constructing and
enumerating them also has polynomial cost. Every vertex encodes
a pair $(x,f(x))$ in the same palette of size $p^2$ and uses the
same sequence of bounds starting from $N_0$. The bounds determine
when to stop.
\end{proof}

Each color class induces a $K_s$-free graph in the original graph. Indeed, a clique on $s$ vertices would make every pair an edge of the conflict graph, so its vertices would receive distinct colors. A largest class therefore contains at least $|S|/C$ vertices and gives the Ramsey exponent with the original independence bound.

For local access we use all of $S$. Take the union of the induced graphs on the color classes, with no edges between different classes. This keeps the decoder for $S$ and requires only two local color computations in addition to the original edge test. An independent set has at most the original independence bound in each class, hence at most $C$ times that bound in total. Using one largest class would also require a way to index its vertices.

\begin{theorem}[Quantified critical construction]\label{thm:quantified}
For a critical interpolating flag family satisfying \eqref{eq:endpoint-condition},
take the degree bound $e$ from Proposition~\ref{prop:substituted-degree} and set
\begin{equation}\label{eq:quantitative-constants}
D=(2e)^{b\binom s2},\qquad
\Delta=s m(s-1)D,\qquad\Lambda=s(s-1)D,\qquad C=(16(\Lambda+1))^{16}.
\end{equation}
For every $h$ with $\gcd(b,h)=1$ and $q=2^h>2\Delta$, there is a deterministically specified $K_s$-free graph on $q^{m-1}(q-\Delta)$ indexed vertices. Its independence number is at most $Cd^{bh}$ in the single-stage family and $Cd^k\rho^{bh}$ in the two-stage family. The equations used in preprocessing depend only on the construction parameters. With those parameters fixed, each local query has deterministic bit cost polynomial in $h$.
\end{theorem}

\begin{proof}
Take $X$ to be the closure of the clique incidence in the space of
coefficients and labels. For the two-stage family use the open
restriction to distinct $X$-tuples in all factors from
Proposition~\ref{prop:two-stage}. It contains every actual clique
when $E_q$ is a field. The incidence dimension is at most $L+m$,
and Proposition~\ref{prop:substituted-degree} bounds the cumulative
degree of the closure by $D$. Apply Lemma~\ref{lem:finite-fibers},
specialize its polynomial and decode the labels in its complement.
The conflict degree bound lets us apply Lemma~\ref{lem:color}.
The estimate within each color is the original rank bound. Summing over at most $C$ classes gives the two stated inequalities. For a local query
we only use the polynomial evaluations, root lists, membership
queries and the local color computation.
\end{proof}

For example, set $s=16,d=13,b=5,m=38$ in the single-stage construction. Substituting $e=16$ in \eqref{eq:quantitative-constants} gives $D=2^{3000}$ and $2\Delta=18\,240D$. The inequality $2^h>2\Delta$ starts at $h=3015$, whereas we also need $\gcd(h,5)=1$. Thus the first field degree for which these estimates guarantee the graph is $h=3016$.

\begin{corollary}\label{cor:growing-range}
The single-stage construction at $d=s$ gives a graph whenever $t$ exceeds a sufficient threshold $t_0(s)\le 2^{O(s^2(\log s)^2)}$. In the range $s\to\infty$, $s^2(\log_2s)^2=o(\log_2t)$, its order yields
\[
 R(s,t)\ge t^{(1/2-o(1))s/\log_2s}.
\]
The procedure is deterministic. Its local polynomial bound in the extension degree is a fixed-$s$ bound, with preprocessing depending on $s$.
\end{corollary}

\begin{proof}
Take the single-stage family with $d=s$ and $\mu=s/2$.
For even $s$, its reduced form has $b=1,m=s/2$. For odd $s$ it has
$b=2,m=s$. With $e=s$, the logarithms of $D,\Delta,C$ are all
$O(s^2\log s)$. Consequently there is an $h\equiv1\pmod b$ with $2^h>2\Delta$ and $h=O(s^2\log s)$, since $b$ is at most two. Choose the largest $h\equiv1\pmod b$ with
$Cd^{bh}\le t-1$. When this degree is in the available range,
the gap between consecutive choices is $b$ and the density
of retained labels is at least $1/2$. These give
\[
\log_2|V|\ge\frac{s}{2\log_2s}\bigl(\log_2(t-1)-\log_2C\bigr)-mb-1.
\]
To allow the first degree in this progression it suffices to take
$\log_2(t-1)\ge\log_2 C+bh\log_2s$ at that degree. The right-hand side is $O(s^2(\log s)^2)$. Along the stated sequences, $\log_2 C=o(\log_2t)$ and $mb+1=o(s\log_2t/\log_2s)$. The displayed lower bound on $|V|$ now gives the conclusion because the graph has no $s$-clique and $\alpha\le t-1$.
\end{proof}

\section{Dimension bounds for ordered flag tuples}

\subsection{Closed sets, open conditions and dimension}
We take the ground field to be $\overline{\F}_2$ for the binary construction. A system of polynomial equations defines a Zariski closed set. Requiring a polynomial to be nonzero gives an open set. Intersections of a closed set with an open set are locally closed, and their finite unions are constructible. The incidences are closed conditions. To fix a matrix rank $u$, we impose the vanishing of every $(u+1)$-minor and the nonvanishing of at least one $u$-minor.

For a constructible set, dimension means the dimension of its closure. For a closed affine set it is the largest of the transcendence degrees of the function fields of its irreducible components over the ground field. In the counts below we use additivity of dimension for products and the fact that a polynomial image has closure of dimension at most that of the source. Local parameter counts use finite covers by charts. The assertions on constructible images and fiber dimension are those of \cite[Chapter I]{Hartshorne}.

Project $xy=1$ onto the first coordinate. We obtain all nonzero scalars and hence an image of dimension one, dense in $\A^1$. Omitting a point is therefore insufficient to omit it from the image closure. We use a strict dimension bound when it is available and a separate argument for the diagonal at equality. A generic choice in these arguments is taken in a specified nonempty Zariski open set.

\begin{lemma}
\label{lem:complete-flag-bound}
For the ordered flag tuples with $P_i\subseteq H_j$ whenever $i<j$, one has
\[
 \dim\Ccal_{s,d}\le s(d-2)+\binom d2.
\]
\end{lemma}

\begin{proof}
Let $R_j=\spanop(P_1,\ldots,P_j)$. The incidence conditions give
$R_j\subseteq H_j$, so the distinct $R_j$ form a chain of nonzero
proper subspaces. Extend the chain to a complete flag
\[
F_1\subset F_2\subset\cdots\subset F_{d-1},\qquad\dim F_u=u.
\]
To count complete flags, choose $F_1$ as a point of $\PP^{d-1}$, then the line $F_2/F_1$ in the quotient, and continue. The dimensions add to $\sum_{u=1}^{d-1}(d-u)=\binom d2$. Fix a complete
flag and indices $u_j$ with $R_j=F_{u_j}$. If we allow all pairs
with $P_j\subseteq F_{u_j}\subseteq H_j$, the choices at position
$j$ form
\[
\PP^{u_j-1}\times\PP^{d-u_j-1},
\]
of dimension $d-2$. For each fixed index sequence, the space of
these choices over the complete flags has local product charts
from bases adapted to the flag. Its dimension is
$\binom d2+s(d-2)$. Forgetting the complete flag cannot raise
the dimension, and the finitely many index sequences cover all
ordered configurations.
\end{proof}

\end{document}